\documentclass{article}

\usepackage{times}
\usepackage{graphicx}
\usepackage{subfigure}

\usepackage{natbib,cancel,mdframed}

\usepackage{amsmath,amsthm,amssymb,color,wasysym}
\newcommand{\E}{\mathbb{E}}
\newcommand{\R}{\mathbb{R}}
\newcommand{\GMMD}{\mathrm{gMMD}}
\newcommand{\MMD}{\mathrm{MMD}}
\newcommand{\ED}{\mathrm{ED}}
\newcommand{\EED}{\mathrm{eED}}

\newtheorem{lemma}{Lemma}
\newtheorem{proposition}{Proposition}
\newtheorem{theorem}{Theorem}

\newtheorem{remark}{Remark}

\usepackage{hyperref}
\usepackage[margin=1.25in]{geometry}

\usepackage{phaistos}

\allowdisplaybreaks

\begin{document}
\title{ Adaptivity and Computation-Statistics Tradeoffs for Kernel and Distance based High Dimensional Two Sample Testing }

\author{\\
Aaditya Ramdas$^{12}$\\
\texttt{aramdas@cs.cmu.edu}\\
\hspace{4in} 
\and \\
Sashank J. Reddi$^{2}$\\
\texttt{sjakkamr@cs.cmu.edu}\\
\and \\
Barnab\'{a}s P\'{o}czos$^2$\\
\texttt{bapoczos@cs.cmu.edu} \\
\hspace{2in} 
\and \\
Aarti Singh$^2$ \\
\texttt{aarti@cs.cmu.edu} \\
\and \\
Larry Wasserman$^{12}$ \\
\texttt{larry@stat.cmu.edu} \\
\and \\ 
Department of Statistics$^1$ and Machine Learning Department$^2$\\
Carnegie Mellon University\\ 
}

\maketitle

\begin{abstract}
Nonparametric two sample  testing is a decision theoretic problem that involves identifying differences between two random variables without making parametric assumptions about their underlying distributions. We refer to the most common  settings as  mean difference alternatives (MDA), for testing differences only in first moments, and  general difference alternatives (GDA), which is about testing for any difference in distributions. A large number of test statistics have been proposed for both these settings.
This paper connects three classes of statistics - high dimensional variants of Hotelling's t-test, statistics based on Reproducing Kernel Hilbert Spaces, and energy statistics based on pairwise distances. 
We ask the following question - \textit{how much statistical power do popular kernel and distance based  tests for GDA have when the unknown distributions  differ in their means, compared to specialized tests for MDA?}

To answer this, we formally characterize the power of popular tests for GDA like the Maximum Mean Discrepancy with the Gaussian kernel ($\GMMD$) and bandwidth-dependent variants of the Energy Distance with the Euclidean norm ($\EED$) in the high-dimensional MDA regime. We prove several interesting properties relating these classes of tests under MDA, which include  
\begin{enumerate}
\item[(a)] $\EED$ and gMMD have asymptotically equal power; furthermore they also enjoy a free lunch because, while they are additionally consistent for GDA, they have the same power as specialized high-dimensional t-tests for MDA. All these tests are asymptotically optimal (including matching constants) for MDA under spherical covariances, according to simple lower bounds.
\item[(b) ] The power of $\GMMD$ is independent of the kernel bandwidth, as long as it is larger than the choice made by the median heuristic.
\item[(c)]  There is a clear and smooth computation-statistics tradeoff for linear-time, subquadratic-time and quadratic-time versions of these tests, with more computation resulting in higher power.
\end{enumerate}  All three observations are practically important, since point (a) implies that $\EED$ and $\GMMD$ while being consistent against all alternatives, are also automatically adaptive to simpler alternatives, point (b) suggests that the median ``heuristic'' has some theoretical justification for being a default bandwidth choice, and point (c) implies that expending more computation may yield direct statistical benefit by orders of magnitude. 
\end{abstract}


\section{Introduction}

Nonparametric two sample testing (or homogeneity testing) deals with detecting differences between two distributions, given samples from both, without making any parametric distributional assumptions. More formally, given samples $X_1,...,X_n \sim P$ and $Y_1,...,Y_m \sim Q$, where $P$ and $Q$ are distributions in $\mathbb{R}^d$, the most common types of two sample tests involve testing for the following sets of null and alternate hypotheses
\begin{align*}
\text{General difference alternatives \textbf{(GDA)} : ~}~ &H_0: P = Q &~\mbox{~ vs ~}~& H_1: P\neq Q,\\
\text{Mean difference alternatives \textbf{(MDA)} : ~}~ &H_0: \mu_P = \mu_Q &~\mbox{~ vs ~}~& H_1: \mu_P \neq \mu_Q
\end{align*}
where $\mu_P := \E_P X, \mu_Q := \E_Q Y$. This problem has a sustained interest in both the statistics and machine learning literature, due to applications where the sample size might be limited compared to dimensionality, due to experimental or computational costs. For example, it can be used to answer questions in medicine (\textit{is there a difference between pill and placebo?})  and neuroscience (\textit{does a particular brain region respond differently to two different kinds of stimuli?}). 

 We will assume $m=n$ for simplicity, though our results may be extended to the case when $m/(n+m)$ converges to any constant $k \in (0,1)$. A test  $\eta $ is a function from $X_1,...X_n,Y_1,...,Y_n$ to $\{0,1\}$, where we reject $H_0$ when $\eta=1$. We will only consider tests  that have an asymptotic type-I error of at most $\alpha$. Let us call the set of all such tests as
\begin{equation}\label{eq:tests}
[\eta]_{n,d,\alpha} := \{\eta : \R^{n\times d} \times \R^{n\times d} \to \{0,1\}, \E_{H_0} \eta \leq \alpha + o(1)\}.
\end{equation}
In the Neyman-Pearson paradigm for the fixed $d$ setting, a test is judged by its power  $\phi = \phi(n,P,Q,\alpha) = \E_{H_1} \eta$, and we say that such a test $\eta \in [\eta]_{n,d,\alpha}$ is consistent in the fixed $d$ setting when
$$
\E_{H_1} \eta \rightarrow 1, \E_{H_0} \eta \leq \alpha \mbox{ as } n \rightarrow \infty \text{ for any fixed $\alpha > 0$}.
$$
In contrast, we say that a test $\eta \in [\eta]_{n,d,\alpha}$ is consistent in the high-dimensional setting when its power $\phi = \phi(n,d_n,P_n,Q_n,\alpha) = \E_{H_1} \eta$ satisfies
$$
\E_{H_1}\eta \rightarrow 1, \E_{H_0} \leq \alpha \mbox{ as } (n,d) \rightarrow \infty, \text{ for any fixed $\alpha > 0$}
$$
where one also needs to specify the relative rate at which $n,d$ can increase.
The central question being considered in this paper is \textit{``what is the power of tests designed for GDA, compared to those designed for MDA, when the distributions truly differ in their means?''}. We will explain this and other related questions in more detail in Section \ref{sec:open}.

\begin{remark}
The tests considered in this paper have some common properties. All the test statistics $T$ are centered under the null, i.e. $\E_{H_0} T = 0$, dividing the statistic by $\sqrt{var(T)}$ leads to an asymptotically  standard normal statistic  under the null, i.e. $T/\sqrt{var(T)}   \rightsquigarrow N(0,1)$ under $H_0$, where $\rightsquigarrow $ represents convergence in distribution as $n \to \infty$, and hence all tests are of the form:
$$
\eta(X_1,...,X_n, Y_1,...,Y_n) = \mathbb{I}\left( \frac{T}{\sqrt{var(T)}} > z_\alpha \right)
$$
where $z_\alpha$ is the $1-\alpha$ quantile of the standard normal distribution.
\end{remark}

Two-sample testing is a fundamental decision-theoretic problem, having a long history in statistics - for example, the past century has seen a wide adoption of the t-statistic by \citet{hotelling} to decide if two samples have different population means (MDA). It was introduced in the parametric setting for univariate Gaussians,  but it has been generalized to multivariate non-Gaussian settings as well. If $\bar X,\bar Y$ are the sample means, and $S$ is a joint sample covariance matrix, then a statistician using the multivariate $t$-test calculates
$$
T_H := (\bar X - \bar Y)^T S^{-1} (\bar X - \bar Y)
$$
and the test is $\mathbb{I}(T_H/\sqrt{Var(T_H)} > t_\alpha)$ where $t_\alpha$ is chosen so that $\E_{H_0} \eta \leq \alpha + o(1)$).  $T_H$ is consistent for MDA whenever $P,Q$ have different means, and further, it is known to be the ``uniformly most powerful'' test when $P,Q$ are univariate Gaussians under fairly general assumptions \citep{kariya81,simaika41,anderson58,salaevskii71}.

 In a seminal paper by \cite{bs}, the authors proved that $T_H$ has asymptotic power tending to $\alpha$ in this high-dimensional setting (as discussed in the next section), motivating the study of alternative test statistics.  Despite their increasing popularity and usage,
  many interesting questions remain unanswered, as will be discussed in Section \ref{sec:open} and partially answered in this paper. 
  This paper deals with (moderately) high-dimensional and nonparametric two-sample testing, where $d$ can grow polynomially with $n$, and there are no explicit parametric assumptions on $P,Q$. In Section \ref{sec:expt}, we experimentally validate our claims for a variety of distributions, even at quite small sample sizes and dimensions. This shows that the asymptotics accurately describe even finite sample behavior of these tests.

  \paragraph{Paper Outline.} The rest of this paper is organized as follows. In Section \ref{sec:tests}, we introduce three classes of tests in the literature - Hotelling-based tests for MDA, and kernel-based and distance-based tests for GDA, and we discuss related open questions in Section \ref{sec:open}. In Section \ref{sec:main}, we prove that three of the most popular tests (one from each class) have the same asymptotic power for MDA, showing the free adaptivity of GDA-based tests for the simpler MDA problem. In Section \ref{sec:lower}, we show that all these classes of tests are optimal for MDA under the diagonal covariance setting, by adapting a lower bound from the normal means problem.  Section \ref{sec:compstat} discusses computation-statistics tradeoffs, where we compare the power of linear-time, sub-quadratic time and quadratic-time versions of these tests. In Section \ref{sec:expt}, we run experiments and discuss some practical implications of this work. We  end with the proofs in Section \ref{sec:proofs}. 

\paragraph{Notation} We use the standard $o,o_P,O_P$ notation extensively. Also, for two non-random sequences $A_n,B_n$, $A_n = \Omega(B_n)$ is the negation of $A_n = o(B_n)$,  $A_n = \omega(B_n)$ is the negation of $A_n = O(B_n)$,  and $A_n \asymp B_n$ to mean  $A_n = B_n(c+o(1))$ for some absolute constant $c$. $Tr()$ is the trace of a (square) matrix and $Tr^k()$ is the $k$-th power of the trace. $\circ$ is the elementwise or Hadamard product, $Ts()$ refers to the total sum of all the elements of a matrix, $e_i$ is the $i$-th standard basis vector, $1$ is the vector of ones. $\rightsquigarrow$ is convergence in distribution, and $\mathbb{I}(\cdot)$ is a 0-1 indicator function.

\section{Hotelling-based MDA Tests and Kernel/Distance-based GDA tests}\label{sec:tests}

\textbf{Tests for MDA}. As mentioned in the introduction, \cite{bs} prove that Hotelling's $T_H$ has power tending to $\alpha$ (this is called trivial power),  when $(n,d)\rightarrow \infty$ with $d/n \rightarrow 1-\epsilon$ for small $\epsilon$, explained by the inherent difficulty of accurately estimating the $O(d^2)$ parameters of $\Sigma^{-1}$ with very few samples ($S^{-1}$ is not even defined if $d>n$ and is badly conditioned if $d$ is of similar order as $n$). To avoid this problem, they  proposed to use the test statistic
$$
T_{BS} := \|\bar X - \bar Y\|^2 - \mathrm{tr}(S)/n
$$
and showed that it has non-trivial power whenever $d/n \rightarrow c \in (0,\infty)$. An important precursor to this nonparametric work of \cite{bs} is that of \cite{dempster58} who proposed a high-dimensional t-test for Gaussians.
\cite{sd} and \cite{skk13} proposed to instead use $\mathrm{diag}(S)^{-1}$ instead of $S^{-1}$, in $T_H$,
and showed its advantages in certain settings over $T_{BS}$ (specifically its scale invariance, i.e. invariance when the data is  rescaled by a diagonal matrix, gives it an advantage when the covariance matrices are diagonal but non-spherical).

In another extension of $T_{BS}$ by \cite{cq}, henceforth called CQ, the authors proposed a  variant of $T_{BS}$ of the form
\begin{eqnarray*}
T_{CQ} &:=& \frac1{n(n-1)}\sum_{i\neq j=1}^n X_i^T X_j
+ \frac1{n(n-1)}\sum_{i\neq j=1}^n Y_i^T Y_j
 -  \frac{2}{n^2}\sum_{i,j=1}^n X_i^T Y_j,
\end{eqnarray*}
analyzing its power for MDA when the covariances of $X,Y$ are also unequal and without explicit restrictions on $d,n$, but rather in terms of conditions stated in terms of $n,\Sigma$ and mean difference $\delta := \mu_P - \mu_Q$. We will return to these conditions later in this paper, since we will use assumptions of similar flavor.

Note that $\E[T_{CQ}] = \mu_P^T \mu_P + \mu_Q^T \mu_Q - 2 \mu_P^T \mu_Q = \|\mu_P - \mu_Q\|^2$, and hence $T_{CQ}$ is an unbiased estimator of $\|\mu_P - \mu_Q\|^2$. In this paper, instead of using $T_{CQ}$ directly, we will analyze a minor variant, which is a U-statistic:
\begin{eqnarray}
U_{CQ} &:=& \frac1{n(n-1)} \sum_{i \neq j = 1}^n h_{CQ}(X_i,X_j,Y_i,Y_j) \nonumber\\
\text{where } h_{CQ}(X,X',Y,Y') &:=& X^T X + Y^T Y - X^T Y' - X'^T Y. \label{eq:hcq}
\end{eqnarray}
$T_{CQ}$'s difference from $U_{CQ}$ is only in the third term, and this difference is asymptotically vanishing, making the asymptotic properties of $U_{CQ}$ (especially its power) identical to $T_{CQ}$, and its usage is only for technical convenience.

There is also a large literature on the so-called parametric Behrens-Fisher problem, which is a parametric MDA problem where the distributions are Gaussian and heteroskedastic, and also the nonparametric Behrens-Fisher problem that deals with MDA when $P,Q$ are nonparametric mean-scale families, in the univariate and multivariate settings. See \cite{belloni} and \cite{lopes11} for recent such works, and references therein. Another related line of work analyzes the setting where $p$ could be exponentially larger than $n$ but assuming some kind of sparsity (say in the mean difference); see \cite{cai14} for such an example.


\textbf{Tests for GDA}. It is well known that the Kolmogorov-Smirnov (KS) test by \cite{kolmogorov33} and \cite{smirnov48} involves differences in empirical CDFs. The KS test, the related Cramer von-Mises criterion by \cite{cramer28} and \cite{vonmises28}, and Anderson-Darling test by \cite{anderson1952asymptotic} are very popular in one dimension, but their usage has been  more restricted in higher dimensions. This is mostly due to the curse of dimensionality involved with estimating multivariate empirical CDFs. While there has been  work on generalizing these popular one-dimensional to higher dimensions, like \cite{bickel69}, these are seemingly not the most common multivariate tests. Some other examples of univariate tests include rank based tests as covered by the book \cite{lehmann06ranks} and the runs test by \cite{waldwolfowitz40}, while some interesting multivariate tests include spanning tree methods by \cite{friedmanrafsky79}, nearest-neighbor based tests by \cite{schilling86} and \cite{henze88}, and the ``cross-match'' tests by \cite{rosenbaum}. Most of these have been proved to be consistent in the fixed $d$ setting, but not much is known about their power in the high-dimensional setting.

One popular class of tests for the multivariate GDA problem that has emerged over the last decade, are kernel-based tests introduced in parallel by \cite{alba08} and \cite{mmd06}, and expanded on in \cite{mmd}. The \textit{Maximum Mean Discrepancy} between $P,Q$ is defined as
$$
\MMD(H_\kappa,P,Q) := \max_{\|f\|_{H_\kappa} \leq 1} \E_P f(x) - \E_Q f(y)
$$
where $H_\kappa$ is a Reproducing Kernel Hilbert Space associated with Mercer kernel $k(\cdot,\cdot)$, and $\{f: \|f\|_{H_\kappa} \leq 1\}$ is its unit norm ball. It is easy to see that $\MMD \geq 0$, and also that $P=Q$ implies $\MMD=0$. For the converse, \cite{mmd06} show that under fairly general conditions involving $H_\kappa$ or equivalently $\kappa$, the equality holds iff $P=Q$. The authors prove that
$$
\MMD(H_\kappa, P,Q) = \|\E_P \kappa(x,.) - \E_Q \kappa(y,.)\|_{H_\kappa}.
$$
This gives rise to a natural associated test, that involves thresholding the following U-statistic, an unbiased estimator of $\MMD^2$:
\begin{eqnarray}
\MMD^2_u(k(\cdot , \cdot)) &:=& \frac1{n(n-1)}\sum_{i\neq j}^n h_\kappa(X_i,X_j,Y_i,Y_j)\nonumber\\
\text{where } h_\kappa(X,X',Y,Y') &:=& \kappa(X,X') + \kappa(Y,Y') - \kappa(X,Y') -\kappa(X',Y). \label{eq:hk}
\end{eqnarray}
Note once again that we can form a $\GMMD$ statistic having 3 summations like $T_{CQ}$, but for technical convenience we mimic the form of the U-statistic $U_{CQ}$, the asymptotic properties of both being the same. Note that $U_{CQ}$ is just the $\MMD$ when we use the linear kernel $k(a,b)=a^T b$.
 The most popular kernel for GDA is the Gaussian kernel with bandwidth parameter $\gamma$, leading to the test statistic that we henceforth call $\GMMD$:
\begin{eqnarray*}
\GMMD^2_\gamma &:=& \MMD^2_u(g_\gamma(\cdot,\cdot))\\
\mbox{where ~}~ g_\gamma(a,b) &:=& \exp\left(-\frac{\|a - b \|_2^2}{\gamma^2} \right).
\end{eqnarray*}
Apart from the fact that the population $\GMMD^2(P,Q) = 0$ iff $P=Q$  the other fact that makes this a useful test statistic is that its estimation error, i.e. the error of $\MMD^2_u$ in estimating $\MMD^2$, scales like $1/\sqrt n$, independent of $d$; see \cite{mmd} for a detailed proof of this fact. This is unlike the KL divergence, for example, which is $0$ iff $P=Q$ but is hard to estimate in high-dimensions. However, it was recently argued in \cite{powermmd} that the study of estimation error covers only one side of the story, and that test power still degrades with $d$ even if estimation error does not.

A related but different class of tests are distance-based ``energy statistics'' as introduced in parallel by \cite{baringhausfranz04} and  \cite{energydistance}, and generalized to some kinds of metrics, denoted $\rho$, for a related independence testing problem, by \cite{lyons}. The test statistic is called the \textit{Cramer statistic} by the former paper but we use the term \textit{Energy Distance} as done by the latter, and once more, we study the U-statistic form:
\begin{eqnarray}
\ED_u(\rho(\cdot,\cdot)) &:=&   
\frac1{n(n-1)}\sum_{i\neq j}^n h_\rho(X_i,X_j,Y_i,Y_j)\nonumber \\
\text{where } h_\rho(X,X',Y,Y') &:=& \rho(X,Y') + \rho(X',Y) - \rho(X,X') -\rho(Y,Y'). \label{eq:hd}
\end{eqnarray}
The most popular or ``default'' choice within this class (the only one studied by both sets of authors who introduced it) is the Energy Distance with the Euclidean distance, henceforth called $\EED$, defined as
\begin{eqnarray*}
\EED_u &:=& \ED_u(e(\cdot,\cdot))\\
\mbox{where ~}~ e(a,b) &:=& \|a - b \|_2.
\end{eqnarray*}

Appropriately thresholding $\GMMD^2_u$ and $\EED_u$ leads to tests that are consistent for GDA in the fixed $d$ setting against all fixed alternatives where $P \neq Q$  (and some local alternatives, i.e. alternatives that change with $n$) under fairly general conditions and such results can be found in the associated references. However not much is known about them in the high dimensional regime.

\begin{remark}
This paper will deal largely with $\GMMD$ and $\EED$, because these are the most popular choices for kernel and distance used in practice, but similar inferences can possibly be made about other kernels and distances, using the same proof technique. Similarly, we will focus on $U_{CQ}$, though one may draw similar inferences about $T_{BS}$ and $T_{SD}$ and their corresponding GDA variants.
\end{remark}


\section{Open Questions and Summary of Results}\label{sec:open}

The test statistics  for MDA, like $U_{CQ}, T_{BS}, T_{SD}, T_{H}$ have all been analysed in the high-dimensional setting. However, there is presently poor understanding of $\GMMD$ and $\EED$ in high dimensions. Below we list some of these open questions (along with explanations) that we are going to answer in this paper, followed by our partial answers to these questions. 

\paragraph{Q1.} How can one characterize the power of nonparametric tests like $\GMMD$ and/or $\EED$ in high dimensions, either for GDA or MDA?\\

\textit{Explanation [Q1].} In the fixed $d$ setting, $\GMMD$ and $\EED$ are well understood, and their null and alternate distributions  are given in \cite{mmd} and \cite{energydistance} respectively. However, their behavior in high dimensions seems to be essentially unanswered in the current literature. A general characterization of power is impossible since $P,Q$ could be different yet arbitrarily similar to each other (see Section 3.2 of \cite{mmd} for a formal statement and proof of this claim). Due to this reason, one is somewhat restricted to trying to characterize the power in limited settings. For example, one can hope to characterize the power by parameterizing the problem in terms of the smallest moment in which $P,Q$ differ.  

\textit{Result [Q1].} One way that we propose to analyze them is to consider two nonparametric distributions $P,Q$ that \textit{only} differ in one specific moment and see how much power $\GMMD$ or $\EED$ have to identify this difference and reject the null. As a first step, this paper will characterize their power for MDA, when $P,Q$ differ only in their first moment. 

\paragraph{Q2.} How does the choice of bandwidth parameter $\gamma$ affect power of $\GMMD_u^2$, for GDA or MDA?\\

\textit{Explanation [Q2].} The most popular choice of bandwidth is the ``median heuristic'' where it is chosen as the median Euclidean distance between all pairs of points (see \cite{learningkernels}). However, the effect of this choice on test power is unclear. \cite{optimalkernel} also make suggestions for choosing the bandwidth parameter, but only for the linear-time $\GMMD^2_l$ (see Section \ref{sec:compstat}), and also with guarantees only in the fixed $d$ setting. Hence the study of how the kernel bandwidth affects power is a work in progress in the current literature. For any fixed $\gamma$, consistency for GDA was proved in \cite{mmd06}; further, the power of $\GMMD^2_u$ against any fixed GDA alternative was also explicitly derived in the fixed $d$ setting to be $\Phi(\sqrt n)$, ignoring constants, where $\Phi$ is the Gaussian CDF. Notice that consistency of the $\GMMD$ test for any \textit{fixed} $\gamma$ is in stark contrast to using Gaussian kernels for density estimation, where we must let the bandwidth go to zero with increasing $n$, and hence the $\GMMD$ statistic does not behave in the same way as the L2-distance between kernel density estimates, as done in \cite{andersonhall94}.

\textit{Result [Q2].} In Section \ref{sec:main}, we prove that the power of $\GMMD_u^2$ does not depend on the bandwidth parameter $\gamma$, as long as $\gamma$ is chosen to be asymptotically larger than the choice made by the aforementioned median heuristic.

\paragraph{Q3.} Can one directly compare the power of $\EED$ and $\GMMD$ for GDA or MDA? Is one of them more powerful than the other? \\

\textit{Explanation [Q3].}  \cite{hsiceqdcov} describes connections between kernel and distance based tests for independence testing. Informally speaking, there is a near one-to-one correspondence between the class of kernels and distances for which such tests make sense. However, while there is \textit{some} metric/semimetric that corresponding to Gaussian kernel $g$, that metric/semimetric is not the Euclidean distance $e$ (and vice versa).  $\EED$ seems to be more popular in the statistics literature, and $\GMMD$ in machine learning - it is of practical importance to both fields to know how one should choose between $\EED$ and $\GMMD$. 

\textit{Result [Q3].} In Section \ref{sec:main}, we show that (under fairly general conditions) $\GMMD$ and $\EED$ have asymptotically equal power for MDA, both in theory and practice.

\paragraph{Q4.} How do the powers of tests for GDA compare to tests for MDA, when (unknown to us) $P,Q$ actually differ in only their means?\\

\textit{Explanation [Q4].}  Given a nonparametric two-sample testing problem, one generally does not know if the distributions differed in their means or not. If they did differ in their means, presumably the former statistics may perform worse than the latter, since the latter are designed specifically for that purpose, and can concentrate all their power in detecting first moment differences. But how much worse? What is the price one must pay for the extra generality of $\GMMD$ and $\EED$? One of the main questions considered in this paper is actually one of comparing the powers of $\EED,\GMMD$ and $U_{CQ}$. 

\textit{Result [Q4].} In Section \ref{sec:main}, we prove that one does not pay any price for the generality of $\GMMD_u^2, \EED_u$ (they enjoy a ``free lunch'') - 
$\GMMD^2_u$ and $\EED_u$ have the same power as $U_{CQ}$ against MDA in high dimensions, both in theory and practice, even though $\GMMD^2_u$ and $\EED_u$ are also consistent against GDA whereas $U_{CQ}$ is not. We would like to note that this result has actually been observed in practice, but seemingly not been explicitly acknowledged or conjectured. Figures 1 and 4 of \cite{baringhausfranz04}  are quite convincing for $\EED$, and the authors explicitly point this out in their experiments and conclusion sections, while Figures 3 and 4 of \cite{lopes11} also show same phenomenon for $\GMMD$, though the latter authors do not comment on their experimental observation. As far as we know, this paper has the first rigorous justification of such a phenomenon.

\paragraph{Q5.} How does computation affect power in high dimensions?\\ 

\textit{Explanation [Q5].} A final question we consider is the relationship between computation and power. Noting that $\GMMD_u^2$ takes quadratic time i.e. $O(n^2)$ to compute, \cite{mmd} and \cite{btest} introduce linear-time and block-based subquadratic-time statistics $\GMMD^2_l$ and $\GMMD_b^2$. The main related work in this regard is \cite{powermmd2}, which analyses a linear-time version of $\GMMD^2_l$ in the high-dimensional setting. We will discuss this last question in detail in Section \ref{sec:compstat}.

\textit{Result [Q5].}  In Section \ref{sec:compstat}, we show that expending more computation yields a direct statistical benefit of higher power; there is clear and smooth statistics-computation tradeoff for a family of earlier proposed sub-quadratic and linear time (kernel) two sample tests.

\paragraph{Q6.} What are the lower bounds for two sample testing in high dimensions?\\

\textit{Explanation [Q6].} We have not seen any lower bounds for the two sample testing problem in the literature, and definitely none for the high dimensional setting, even under MDA.

\textit{Result [Q6].} In Section \ref{sec:lower}, we prove tight lower bounds for two-sample testing under MDA, for the case of diagonal covariance, which show that all three tests are optimal in this setting, even including constants. 




\section{Adaptivity of $\GMMD$ and $\EED$ to MDA}\label{sec:main}

This section will aim to provide some answers to questions Q1-4.
Our main assumptions are inspired by those in \cite{bs} and \cite{cq}, and related followup papers.

\paragraph{[A1]} \textbf{Model.} $X_i = \Gamma Z_{1i} + \mu_P$ and $Y_i = \Gamma Z_{2i} + \mu_Q$ for $i=1,...,n$ where $Z_{1i},Z_{2i}$ are $k$-dimensional independent  zero mean, identity covariance random variables and $\Gamma$ is a $d \times D$ unknown full-rank deterministic transformation matrix for some $D \geq d$ satisfying $\Gamma \Gamma' = \Sigma$ (hence the  $d \times d$ population covariance $\Sigma$ is full-rank).  Denote the mean difference as $\delta := \mu_P - \mu_Q $.

\begin{remark} Assumption [A1] implies that $X,Y$ have means $\mu_1,\mu_2$ and covariances $\Sigma$, like in \cite{bs}. We do not assume that $X,Y$ have different covariances $ \Sigma_1,\Sigma_2$ like in \cite{cq}. The reason for this choice is as follows. $\GMMD$ and $\EED$ can detect differences in distributions $P,Q$ that occur in any finite moment. 
For example, by Bochner's theorem (see \cite{Rudin62}), the population quantity $\GMMD^2$ is precisely (up to constants)
$$
\int_{\R^d} \vert  \varphi_X(t) - \varphi_Y(t) \vert ^2 e^{-\gamma^2 \|t\|^2} \mathrm{d}t
$$
where $\varphi_X(t)=\E_{x \sim P}[e^{-it^Tx}]$ is the characteristic function of $X$ at frequency $t$ (similarly $\varphi_Y(t)$), and the population $\EED$ is precisely  (up to constants)
$$
\int\displaylimits_{(a,t) \in S^{d-1} \times \R}  [F_X(a,t) - F_Y(a,t)]^2 \mathrm{d}a\  \mathrm{d}t
$$
where $F_X(a,t) = P(a^T X \leq t)$ (similarly $F_Y(a,t)$) is the population CDF of $X$ when projected along direction $a$ and $S^{d-1}$ is the surface of the $d$ dimensional unit sphere; see \cite{energydistance} for a proof.
Because of this, $\GMMD$ and $\EED$ are sensitive to differences in second (and higher) moments of distributions. To analyze their power against MDA, it makes sense to nullify all other sources of signal like $\|\Sigma_1 - \Sigma_2\|_F^2$ that might alter the power of $\GMMD$ or $\EED$. 
\end{remark}


\paragraph{[A2]} \textbf{Moment assumption.} Each of the $D$ coordinates of $Z_{1i}$ and $Z_{2i}$ have $m \geq 8$ moments, each moment being a finite constant. For all $i=1,...,n$ and $s=1,2$, we have $\E(Z_{si1}^{\alpha_1}Z_{si2}^{\alpha_2},...,Z_{siD}^{\alpha_D}) = \E (Z_{si1}^{\alpha_1})\E(Z_{si 2}^{\alpha_2})...\E(Z_{si D}^{\alpha_D}) $ for all $\sum_{j=1}^D \alpha_j \leq 8$. 

\begin{remark}
Assumption [A2] was made in essentially the same form in \cite{bs} and \cite{cq}. Some of our calculations explicitly involve how much these moments deviate from those of a standard Gaussian.
We show in Section \ref{sec:expt} that many of our results hold experimentally for a variety of non-Gaussian distributions.
\end{remark}

\paragraph{[A3]} \textbf{Fairly good conditioning of $\Sigma$.} (a) We assume that $Tr(\Sigma^{2k}) = o(Tr^2(\Sigma^k))$ for $k=1,2$. (b) We also assume that $Tr(\Sigma) \asymp d$
and for $S_i \in \{ X_i,Y_i\}$,  the average $\|S_i - S_j\|^2/d$ exponentially concentrates around its expectation, i.e.
$$
P\left( \Bigg\vert   \frac{\|S_i - S_j\|^2}{d}  - \frac{\E\|S_i - S_j\|^2}{d} \Bigg\vert  > d^{-\nu} \right) \to 0 \text{ exponentially fast in (some polynomial of) $d$.}
$$
for some $\nu = \nu(\Sigma,m) \in (1/3, 1/2]$.

\begin{remark}
Assumption [A3] essentially means that $\Sigma$ is fairly well conditioned, and was also made in the aforementioned earlier works. To see this, note that if $\Sigma = \sigma^2 I$ then the conditions reduce to requiring $d = o(d^2)$. If all the eigenvalues of $\Sigma$ are bounded, this assumption is still met. When $\Sigma$'s eigenvalues are not bounded, this condition will be satisfied as long as $\Sigma$ is not terribly conditioned. This assumption is discussed in detail with several nontrivial examples in \cite{cq}. Similarly, $\nu(\Sigma,m)$ reflects the conditioning of $\Sigma$, and the number $m$ of moments of $S$. In the best case, with $d$ independent coordinates i.e. identity covariance $\Sigma =I$ and infinite moments, $\nu(\Sigma,m) = 1/2$. As we assume fewer moments or as we deviate away from diagonal covariance to more ill-conditioned matrices, $\nu(\Sigma,m)$ strays away from half, but we assume it is fairly well-conditioned, being at least $1/3$. We think that some such good conditioning is necessary for our theorems to hold, but that the scalar $1/3$ can be lowered.
\end{remark}


\paragraph{[A4]} \textbf{Low signal strength.} $\|\delta\|^2 = o \left( \min\left\{\frac{Tr^2(\Sigma)}{Tr(\Sigma^2)} \lambda_{\min}(\Sigma), \frac{Tr(\Sigma)}{d^\nu} \right\} \right)$ and $\delta^T \Sigma^k \delta = o(Tr(\Sigma^{k+1}))$ for $k=0,1,2,3$.

\begin{remark}
First recall that we assumed $\Sigma$ is full rank in Assumption [A1], so $\lambda_{\min}(\Sigma)>0$. Assumption [A4] essentially means that the signal strength is not very \textit{large} relative to the noise. For example, when $\Sigma = \sigma^2 I$, the assumption requires that $\|\delta\|^2/\sigma^2 = o(\sqrt d)$. Indeed, it more generally implies that $\|\delta\|^2 = o(Tr(\Sigma))$\footnote{This holds because $Tr(\Sigma) = Tr(\Sigma^2\Sigma^{-1}) \leq Tr(\Sigma^2)\lambda^{-1}_{\min}(\Sigma)$ by Cauchy-Schwarz inequality that $Tr(A^T B) \leq \|A\|_* \|B\|_{op}$ where $*,op$ refer to the nuclear and operator norms respectively.}. We need this assumption for technical reasons, and we conjecture that our results hold under a weaker assumption. Even in its present form, this is not such a strong assumption since (as we shall see in the theorem statements) if the signal strength is  large then the decision problem becomes too easy and such a regime is rather uninteresting.  Further note that $\delta^T \delta = o(Tr(\Sigma))$ implies, by Cauchy-Schwarz,
\begin{eqnarray*}
\delta^T \Sigma \delta &\leq& \lambda_{\max}(\Sigma)\|\delta\|^2 = o(\lambda_{\max}(\Sigma)Tr(\Sigma)),\\
\delta^T \Sigma^2 \delta &\leq& Tr(\Sigma^2) \|\delta\|^2 = o(Tr(\Sigma^2)Tr(\Sigma)),\\
\delta^T \Sigma^3 \delta &=& o(Tr(\Sigma^3)Tr(\Sigma)) \leq o(Tr(\Sigma^2)Tr^2(\Sigma)).
\end{eqnarray*}
\end{remark}

\paragraph{[A5]} \textbf{High-dimensional setting.}  $n = o(d^{3\nu-1}Tr(\Sigma^2)) = o(\sqrt {d} Tr^2(\Sigma)) = o(d^{2.5})$.

\begin{remark}
Currently, Assumption [A5] is needed only for a technicality in proving our main theorem, and we conjecture that it can be relaxed. 
\end{remark}


As in \cite{cq}, we do not assume that $(n,d) \to \infty$ at any particular rate. 
Instead, we will analyze their behavior in two regimes that have implicit control on $n,d$. For notational convenience, denote
\begin{eqnarray}
\sigma_{n1}^2 &:=& 8\frac{Tr(\Sigma^2)}{n^2} , \label{eq:sigma1}\\
\sigma_{n2}^2 &:=& 8 \frac{\delta^T \Sigma \delta }{n}. \label{eq:sigma2}
\end{eqnarray}
Recalling that $\delta:= \mu_P - \mu_Q$, the first theorem summarizes the power of $U_{CQ}$. 

\begin{theorem}\label{thm:cq}
Under [A1], [A2] and [A3a], 
 $U_{CQ}$ has asymptotic power which equals
\begin{equation}
\phi_{CQ} =  \Phi\left(-\frac{\sqrt{\frac{Tr(\Sigma^2)}{n^2}}}{\sqrt{\frac{Tr(\Sigma^2)}{n^2} + \frac{\delta^T \Sigma \delta}{n}}} \cdot z_\alpha  + \frac{\|\delta\|^2}{\sqrt{8\frac{Tr(\Sigma^2)}{n^2} + 8 \frac{\delta^T \Sigma \delta }{n}}} \right) + o(1)  \label{eq:cqpower}
\end{equation}
where  $\Phi$ is the Gaussian CDF and $z_\alpha$ is the threshold representing the $\alpha$-quantile of the standard Gaussian distribution.
\end{theorem}

This theorem follows from the main result of \cite{cq} for $U_{CQ}$, and hence we do not reproduce it here. There, the  authors prove that $U_{CQ}$ is asymptotically normally distributed with variance $\sigma_{n1}^2 + \sigma_{n2}^2$ under the alternative, and variance $\sigma_{n1}^2$ under the null (with $\Sigma_1=\Sigma_2=\Sigma$ and $n_1 = n_2 = n$ being used by us). This then gives rise to the above expression for the power $\phi$ fairly easily, except that the authors made a small mistake by interchanging $\sigma_{n1}$ and $\sigma_{n2}$ in one crucial expression (confirmed by email correspondence with the authors, summarized in the Appendix Sec. \ref{appsec:cq}). Another minor difference is that we write down the power as a single expression, while \cite{cq} prefer to write them down in the two aforementioned special cases of low and high SNR.

%

\begin{remark}
The null distribution of $U_{CQ}$ is asymptotically Gaussian under MDA in this high-dimensional setting. This is in stark contrast to the fixed-$d$, increasing-$n$ setting, where the null distribution is an infinite sum of weighted chi-squared distributions, due to the properties of degenerate U-statistics (see \cite{serfling}). This seems to have first been proved by \cite{bs} for $T_{BS}$ using a martingale central limit theorem (see \cite{hallheyde}).
\end{remark}

The next theorem summarizes the power of $\GMMD$, which is also one of the main results of the paper.

\begin{theorem}\label{thm:gmmd}
Assume [A1], [A2], [A3], [A4] and [A5],
and let the bandwidth be chosen as $\gamma^2 = \omega (2Tr(\Sigma))$. Then
 $\GMMD_\gamma$ has asymptotic power 
which is independent of $\gamma$, and equals the power of $U_{CQ}$. In other words, the power is
$$
\phi_{\GMMD} = \Phi\left(-\frac{\sqrt{\frac{Tr(\Sigma^2)}{n^2}}}{\sqrt{\frac{Tr(\Sigma^2)}{n^2} + \frac{\delta^T \Sigma \delta}{n}}} \cdot z_\alpha  + \frac{\|\delta\|^2}{\sqrt{8\frac{Tr(\Sigma^2)}{n^2} + 8 \frac{\delta^T \Sigma \delta }{n}}} \right) + o(1)
$$
for all $\gamma^2 = \omega(2Tr(\Sigma))$.
\end{theorem}

The proof of this theorem is covered in Section \ref{sec:proofs}. While one may conjecture a result like the above due to the claims of \cite{noureddine} that the Gaussian kernel often behaves like the linear kernel in high dimensions, their results only hold true when $n \asymp d$ (apart from other differences in assumptions). Further, they also interpret the results rather pessimistically, by saying that these kernels do not provide an advantage in the high-dimensional setting, but we will demonstrate in experiments that when the linear kernel does not suffice (the distributions have the same mean but differ in their variances), then $U_{CQ}$ has trivial power but $\GMMD$'s power tends to one in reasonable scenarios. Of course, more samples are probably needed to detect differences in second moments compared to differences in first moments.
 Hence, we choose to interpret the above result  optimistically --- \textit{not only is $\GMMD$ capable of detecting any difference in distributions, but it also detects differences in means as well as $U_{CQ}$ which is designed to test only mean differences.}

For the purpose of mathematical analysis, we now introduce a family of statistics, for which $\EED_u$ is a special case. These are defined (recalling Eq.\eqref{eq:hd}) as
\begin{eqnarray*}
\EED_\gamma &:=& \ED_u(e_\gamma(\cdot,\cdot)) \\
\text{where } e_\gamma(a,b) &:=& \sqrt{\gamma^2 - 2Tr(\Sigma) + \|a-b\|^2_2}
\end{eqnarray*}
where $\gamma^2 \geq 2Tr(\Sigma)$ is a constant user-chosen bandwidth parameter. Note that 
$$
\lim_{\gamma^2 \to 2Tr(\Sigma)^+} \EED_\gamma = \EED_u
$$

The next theorem summarizes the power of $\EED_\gamma$, in all cases when $\gamma^2 = \omega(2Tr(\Sigma))$.

\begin{theorem}\label{thm:eed}
Assume [A1], [A2], [A3], [A4] and [A5], and let the bandwidth be chosen as $\gamma^2 = \omega (2Tr(\Sigma))$. Then
 $\EED_\gamma$ has asymptotic power 
which is independent of $\gamma$, and equals the power of $U_{CQ}$. In other words, the power is
$$
\phi_{\EED} = \Phi\left(-\frac{\sqrt{\frac{Tr(\Sigma^2)}{n^2}}}{\sqrt{\frac{Tr(\Sigma^2)}{n^2} + \frac{\delta^T \Sigma \delta}{n}}} \cdot z_\alpha  + \frac{\|\delta\|^2}{\sqrt{8\frac{Tr(\Sigma^2)}{n^2} + 8 \frac{\delta^T \Sigma \delta }{n}}} \right) + o(1)
$$
for all $\gamma^2 = \omega(2Tr(\Sigma))$.
\end{theorem}
The proof of this theorem is similar to the proof of Theorem \ref{thm:gmmd}, and hence is briefly covered at the  end of Section \ref{sec:proofs}, after the proof of Theorem \ref{thm:gmmd}.

\begin{remark} We remark on our inability to prove the above theorems for the limiting case of $\gamma^2 \asymp 2Tr(\Sigma)$. The proofs of Theorems \ref{thm:gmmd} and \ref{thm:eed} are based on a Taylor expansion of the $h_\kappa$ and $h_\rho$ respectively (recall Eqs.\eqref{eq:hk},\eqref{eq:hd} for their definition). This leads to a ``dominant'' Taylor term $U_2/\gamma^2$ which is a U-statistic in $h_2$ and  a ``remainder'' term $U_4/\gamma^4$ which is a U-statistic in $h_4$, where
\begin{eqnarray}
h_2(X,X',Y,Y') &=& \|X-X'\|^2 + \|Y-Y'\|^2 - \|X-Y'\|^2 - \|X'-Y\|^2 \label{eq:h2},\\ 
h_4(X,X',Y,Y') &=& \|X-X'\|^4 + \|Y-Y'\|^4 - \|X-Y'\|^4 - \|X'-Y\|^4.
\label{eq:h4}
\end{eqnarray}
One can easily observe that $h_2 = -2 h_{CQ}$ (see Eq.\eqref{eq:hcq}) and hence the behavior of $U_2$ is immediately captured by the behavior of $U_{CQ}$, the most important fact being that $U_2$ is always Gaussian under the null and the alternative (as mentioned after Theorem \ref{thm:cq} and its following remarks). 
When $\gamma^2 = \omega(Tr(\Sigma))$, we prove that $U_4/\gamma^4 = o_P(U_2/\gamma^2)$. However, when $\gamma^2 \asymp 2Tr(\Sigma)$, our results suggest that $U_4/\gamma^4 = O_P(U_2/\gamma^2)$. However, while we know that $U_2/\gamma^2$ is asymptotically Gaussian, we do not know the limiting distribution of $U_4/\gamma^4$, even though we undertake tedious calculations to find the mean and variance of $U_4$. Hence, while this allows us to make arguments about the mean and variance of $\GMMD$ and $\EED$, we cannot make power claims since for that purpose we require knowing the limiting distribution of $U_4$ under the null. While we conjecture that it is indeed Gaussian and simulations support this, the proof is vastly more complicated than for $U_2$ because the number of terms to be controlled in the martingale central limit theorem is larger (by an order of magnitude, as the number of terms grows exponentially). Proving the above theorem statements for the limiting case is an important direction for future work, and may require development of the theory of U-statistics for high dimensional variables. However, for the moment we show a variety of experiments that support our conjecture, implying that the borderline case is probably a technical limitation.
\end{remark}

\subsection{The Special Case of $\Sigma = \sigma^2 I$}

\noindent Though no explicit assumptions are placed on $n,d$ for the above expression (and hence for consistency to hold), for further understanding of the power of these tests, let us consider the situation when $\Sigma = \sigma^2 I$ and define the signal-to-noise ratio (SNR) as
$$
\mbox{SNR  ~}~ \Psi := \frac{\|\delta\|}{\sigma}.
$$
One can think of $\Psi^2$ as the problem-dependent constant, which determines how hard the testing problem is - of course, the larger the SNR, the easier the distributions are to distinguish. Indeed, in the special case of $P,Q$ being spherical Gaussians, $\Psi^2$ is just the KL-divergence between these distributions.
Then, the expression for power from Eq.\eqref{eq:cqpower}  simplifies to
\begin{equation}\label{eq:powerdiag}
\Phi \left( -\frac{\sqrt d}{\sqrt{ d + n \Psi^2}} z_\alpha + \frac{\Psi^2}{\sqrt{8d/n^2 + 8\Psi^2/n}} \right) + o(1).
\end{equation}
We are most interested in the regimes where $\Psi$ is small. Let us define the three regimes as follows:
\begin{eqnarray}
\mbox{Low SNR: ~}~ \Psi &=& o(\sqrt{d/n}),\\
\mbox{Medium SNR: ~}~ \Psi &\asymp& \sqrt{d/n},\\
\mbox{High SNR: ~}~ \Psi &=& \omega(\sqrt{d/n}). \label{eq:highSNR}
\end{eqnarray}

\begin{remark}
We find it worthy to note that the behavior is \textit{different}\footnote{There is a mistake/typo in the paper by \cite{cq}, which causes them to miss this surprising observation. We have confirmed this important typo with the authors, and describe the context of its occurrence in more detail in the Appendix Sec. \ref{appsec:cq}.}  in the low and high SNR regime. Specifically, in the Low SNR regime, the asymptotic power  is
\begin{equation}
\phi_L =  \Phi\left(-z_\alpha + \frac{n\Psi^2}{\sqrt {8d}} \right) \mbox{~ when ~} \Psi = o(\sqrt{d/n})\label{eq:lowpower}
\end{equation}
while in the high SNR regime, the asymptotic power is
\begin{equation}
\phi_H = \Phi(\sqrt n \Psi/\sqrt 8) \mbox{~ when ~} \Psi = \omega (\sqrt{d/n}). \label{eq:highpower}
\end{equation}
The above two rates  match in the Medium SNR regime, yielding a power  $\asymp \Phi(\sqrt d)$.
\end{remark}



\section{Lower Bounds when $\Sigma=\sigma^2 I$}\label{sec:lower}

Here we show that the form of the power achieved in Theorem \ref{thm:cq} is not improvable under certain assumptions. For example, in the case when  $\Sigma = \sigma^2 I$, we can provide matching lower bounds to Eq. \ref{eq:powerdiag} using techniques from \cite{ingstersuslina} designed for Gaussian normal means problem.
The proof relies on the Gaussian approximations of the central and noncentral chi-squared distributions.

\begin{proposition}\label{prop:lower}
Let $G_d(x,0)$ be the cdf of a central chi-squared distribution with $d$ degrees of freedom and $G_d(x,r)$ be the cdf of a noncentral chi-squared distribution with $d$ degrees of freedom and noncentrality parameter $r$. Then as $d\rightarrow \infty$, we have uniformly over $x,r$
\begin{eqnarray}
G_d(x,0) &=& \Phi\left( \frac{x-d}{\sqrt{2d}}\right) + o(1) ,\\
G_d(x,r^2) &=& \Phi\left( \frac{x-d-r^2}{\sqrt{2d + 4r^2}}\right) + o(1) ,\\
G_d(T_{d\alpha},r^2) &=& \Phi \left( \frac{\sqrt{2d}}{\sqrt{2d+4r^2}}z_\alpha - \frac{r^2}{\sqrt{2d + 4r^2}} \right) + o(1) \label{eq:ingster}
\end{eqnarray}
where $T_{d\alpha}$ is $1-\alpha$ quantile cutoff of the $\chi^2_d$ and $z_\alpha$ is the corresponding quantile of the standard normal.
\end{proposition}

\begin{remark}
Our Eq.\eqref{eq:ingster} differs from \cite{ingstersuslina}[Ch 1.3, Pg 13, Eq. 1.14] where the authors applied the additional approximation that $d \to \infty$ with $r$ fixed (or just $d >> r$) to get 
\begin{equation}\label{eq:ingsterapprox}
G(T_{d\alpha},r^2) = \Phi(z_\alpha - \rho^2/\sqrt{2d}) + o(1).
\end{equation}
We do not make this approximation.
\end{remark}

\begin{proof}[Proof of Proposition \ref{prop:lower}]
The first two expressions appear verbatim in \cite{ingstersuslina}[Ch 1.3, Pg 12]. Substituting $x = T_{d\alpha}$ into the second expression yields
$$
G_d(T_{d\alpha},r^2) = \Phi\left( \frac{T_{d\alpha}-d}{\sqrt{2d + 4r^2}} - \frac{r^2}{\sqrt{2d + 4r^2}} \right) + o(1)
$$
The last expression then follows due to the following fact:
\begin{equation}\label{eq:chisq}
\frac{T_{d\alpha} - d}{\sqrt{2d}} = z_\alpha + o(1),
\end{equation}
Eq.\eqref{eq:chisq} holds by the following argument. First note that
$$(\chi^2_d - d)/\sqrt{2d} \rightsquigarrow  N(0,1).$$
Then by definition of $T_{d\alpha}$,
$$
P(\chi^2_d > T_{d\alpha}) \leq \alpha
$$
which then implies
$$
P\left(Z > \frac{T_{d\alpha} - d}{\sqrt{2d}} + o(1) \right) \leq \alpha
$$ for standard normal $Z$. Since we know that $P(Z > z_\alpha) \leq \alpha$, Eq.\eqref{eq:chisq} follows.

\end{proof}

Next, define $S_{d}(\rho) = \{ \delta \in \R^d ~\vert ~ \|\delta\| = \rho \}$ to be the surface of the $d$-dimensional sphere of radius $\rho$. For the normal means problem, we are given $Z \sim N(\delta,I_d)$ and we test $H_0: \delta=0$ against   $H_1: \delta \in S_d(\rho)$. Recalling the definition of $[\eta]_{n,d,\alpha}$ from Eq.\eqref{eq:tests}, we analogously define $[\eta]_{d,\alpha}$ for the normal means problem as the set of all tests from $\R^d \to [0,1]$ with expected type-1 error at most $\alpha$. Define the minimax power at level $\alpha$ as
$$
\beta(\rho,\alpha) := \inf_{\eta \in [\eta]_{d,\alpha}} \sup_{\delta \in S_d (\rho)} \E_\delta \eta.
$$


\begin{proposition}\label{prop:minimax}
Given $Z \sim N(\delta,I_d)$ where $\|\delta\|=\rho$, the minimax power for the normal means problem is
$$\beta(\rho,\alpha) = 1-G_d(T_{d\alpha}, \rho^2) = \Phi\left(-\frac{\sqrt{2d}}{\sqrt{2d+4\rho^{2}}}T_\alpha + \frac{\rho^{2}}{\sqrt{2d + 4 \rho^{2}}} \right) + o(1).$$
\end{proposition}

\begin{proof}
This proposition is \textit{almost} verbatim from Proposition 2.15 of Pg 69 of \cite{ingstersuslina}. Its proof is given in Example 2.2 on pg 51 of \cite{ingstersuslina}, the end of the example yielding the expression for power as $G_d(T_{d\alpha}, \rho^{2})$. The only difference in our proposition statement is that we directly use the expression $G_d(T_{d\alpha}, \rho^{2})$ in Eq.\eqref{eq:ingster} instead of the approximation in Eq.\eqref{eq:ingsterapprox}.

\end{proof}

The above proposition now directly yields a lower bound for two sample testing when $\Sigma = \sigma^2 I$. Let $\mathcal{F}_d(\rho,\sigma) := \{(P,Q): \E_P[X] - \E_Q[Y] \in S_d(\rho), \E[XX^T] - \E[X]\E[X]^T = \E[YY^T] - \E[Y]\E[Y]^T = \sigma^2 I \}$ represent the set of all pairs of $d$-dimensional distributions $P,Q$ whose means differ by $\delta \in S_d(\rho)$ and whose covariances are both $\sigma^2 I$. Define the minimax power at level $\alpha$ as
$$
\beta(\rho,\alpha,\sigma) := \inf_{\eta \in [\eta]_{n,d,\alpha}} \sup_{(P,Q) \in \mathcal{F}_d(\rho,\sigma)} \E_{P,Q} \eta.
$$

\begin{theorem}
Given $X_1,...,X_n \sim N(0,\sigma^2I_d)$ and $Y_1,...Y_n \sim N(\delta,\sigma^2I_d)$, suppose we want to test $\delta=0$ against $\delta \in S_d(\rho)$. Then putting $\Psi := \rho/\sigma$,
the minimax power is
$$
\beta(\rho,\alpha,\sigma)  = \Phi\left(-\frac{\sqrt{d}}{\sqrt{d+n\Psi^2}}T_\alpha + \frac{\Psi^2}{\sqrt{8d/n^2 + 8\Psi^2/n}} \right) +o(1)
$$
\end{theorem}
\begin{proof}
Denote
$$Z = \sum_i \frac{X_i - Y_i}{\sqrt{2}\sigma \sqrt{n}} = \sqrt{n/2}\frac{(\bar X - \bar Y)}{\sigma}.$$
Under the null,
$$Z \sim N(0,I_d)$$
and under the alternate
$$Z \sim N(\delta,I_d)$$
for $\delta \in S_d(\rho')$, where $\rho' = \sqrt{n/2} \rho/\sigma$, i.e. $\rho'^2 = n\Psi^2/2$. Our claim follows by direct substitution into proposition \ref{prop:minimax}. 

\end{proof}

\begin{remark}
This lower bound expression \textit{exactly} matches the upper bound expression in Eq.\eqref{eq:powerdiag}, including matching constants, showing that \textit{all} of the discussed tests are minimax optimal in this setting of $\Sigma = \sigma^2 I$. Even though the current lower bounds can possibly be strengthened to include nondiagonal $\Sigma$, we remark that we have not been able to find even these diagonal-covariance lower bounds in the two sample testing literature, especially which are accurate even to constants.
\end{remark}


\section{Computation-Statistics Tradeoffs}\label{sec:compstat}

In this section we will consider computationally cheaper alternatives to computing the quadratic time $\GMMD^2$ that were suggested in \cite{mmd} and \cite{btest}, namely a block-based $\GMMD^2_B$ and a linear-time $\GMMD^2_L$. While it is clear that $\GMMD^2$ is the minimum variance unbiased estimator (it is a Rao-Blackwellized U-statistic), it is not clear \textit{how much} worse the other options are - if they are only slightly worse, the computational benefits could be worth it if there is a large amount of data. Due to the lack of a high-dimensional analysis in \cite{mmd}, it was inferred that one suffers for cheaper computation with power that is worse, by a constant factor compared to the power of $\GMMD^2$. 
We will show that, for MDA, the power is worse not by constants but by exponents of $n$ (presumably this would only get worse for GDA). At all points, the Assumptions in Section 3 are assumed to hold wherever needed, so that we can proceed directly to comparisons.

Assume that we divide the data into $B = B(n)$ blocks of size $n/B$ with $n/B \rightarrow \infty$. Let $\GMMD^2(b)$ be the $\GMMD^2$ statistic evaluated only on the samples in block $b \in \{1,...,B\}$, and let the block-based MMD be defined as
$$
\GMMD^2_B = \frac1{B}\sum_{b=1}^B  {\GMMD^2(b)}.
$$
We note that this statistic takes  $(n/B)^2 B = n^2/B$ time to compute.

Also, when using $B=n/2$, i.e. using blocks of size just $2$, since $n/B \to \infty$ does not hold, we look at this case separately. This statistic just takes linear-time to compute, since each block $b$ is just of size 2, and we define the linear time MMD as
\begin{equation}\label{eq:lmmd}
\GMMD^2_L = \frac1{n/2}\sum_{b=1}^{n/2}  {\GMMD^2(b)}.
\end{equation}

\begin{theorem}
Under assumptions [A1], [A2], [A3], [A4], [A5] (appropriately holding for $n/B$ points), and the bandwidth is chosen as $\gamma^2 = \omega(Tr(\Sigma))$,  the power of $\GMMD^2_B$ is
\[
\phi_{\GMMD}^B = \Phi\left( \frac{\sqrt B \|\delta\|^2}{{\sqrt{8\frac{B^2 Tr(\Sigma^2)}{n^2} + 8 \frac{B \delta^T \Sigma \delta }{n}}}} - z_\alpha \frac{\sigma_{B1}}{\sigma_B}\right) + o(1).
\]
\end{theorem}


\begin{proof}
 Let $\sigma_{B1}^2$ and $\sigma_{B2}^2$ be as defined in Eqs.\eqref{eq:sigma1},\eqref{eq:sigma2}, but each calculated on $n/B$ points instead of $n$ points, and scaled by $\gamma^4$, i.e.
\begin{eqnarray*}
\sigma_{B1}^2 &:=& 8\frac{B^2 Tr(\Sigma^2)}{\gamma^4 n^2} \\
\sigma_{B2}^2 &:=& 8 \frac{B \delta^T \Sigma \delta }{\gamma^4 n}.
\end{eqnarray*}
 Define $\sigma_B^2 = \sigma_{B1}^2 + \sigma_{B2}^2$. Then from our earlier arguments we have that
 \begin{eqnarray}
\text{ Under $H_0$, ~}~  \GMMD^2(b) &\rightsquigarrow & N(0,\sigma_{B1}^2),\\
\text{ Under $H_1$, ~}~  \GMMD^2(b) &\rightsquigarrow & N(0,\sigma_{B1}^2 + \sigma_{B2}^2).
 \end{eqnarray}

 Hence, the distribution of $\GMMD_B^2$ is $N(0,\sigma_{B1}^2/B)$ under null and $N(\GMMD^2, \sigma^2_B/B)$ under alternative. Hence, from our earlier results it is straightforward to note that
under $H_0$,
$$\sqrt B \frac{\GMMD_B^2}{\sigma_{B1}} \rightsquigarrow N(0,1)$$
 and under $H_1$,
 $$\sqrt{B}\frac{\GMMD_B^2 - \GMMD^2}{\sigma_B} \rightsquigarrow N(0,1).$$

Hence our test statistic will be
$$
T_B := \sqrt B \frac{\GMMD_B^2}{\sigma_1}
$$
with our test being given by $\mathbb{I}(T_B > z_\alpha)$ where $z_\alpha$ is the $\alpha$ quantile cutoff of the standard normal distribution. Note that in practice, we would simply use a studentized statistic by plugging in the estimated $\sigma_1$. Then, the power of this test is
\begin{eqnarray}
P_{H1}\left(\sqrt B \frac{\GMMD^2_B}{\sigma_{B1}} > z_\alpha \right) &=&  P_{H1} \left( \sqrt{B}\frac{\GMMD_B^2 - \GMMD^2}{\sigma_B} > z_\alpha \frac{\sigma_{B1}}{\sigma_B} - \frac{\sqrt B \GMMD^2}{\sigma_B} \right)\\
&=& 1-\Phi\left(z_\alpha \frac{\sigma_{B1}}{\sigma_B} - \frac{\sqrt B \GMMD^2}{\sigma_B} \right)\\
&=& \Phi\left( \frac{\sqrt B \|\delta\|^2}{{\sqrt{8\frac{B^2 Tr(\Sigma^2)}{n^2} + 8 \frac{B \delta^T \Sigma \delta }{n}}}} - z_\alpha \frac{\sigma_{B1}}{\sigma_B}\right).
\end{eqnarray}

\end{proof}

It is again useful to consider the case of $\Sigma = \sigma^2 I$ for some insight, and recall $\Psi = \|\delta\|/\sigma$.  
Specifically, the power  is
\begin{equation}
\phi^B_L =  \Phi\left(\frac{n\Psi^2}{\sqrt {8Bd}} - z_\alpha \right) \mbox{~ when ~} \Psi = o(\sqrt{Bd/n})\label{eq:lowpowerB}
\end{equation}
while in the \textit{very} high SNR regime, the power behaves like
\begin{equation}
\phi^B_H = \Phi(\sqrt n \Psi /\sqrt 8) \mbox{~ when ~} \Psi = \omega (\sqrt{Bd/n}). \label{eq:highpowerB}
\end{equation}

Of course, the above two rates match in the Medium SNR regime. Here we use the italicized \textit{very} because it is a $\sqrt{B}$ times larger SNR requirement than the high SNR regime given in Eq.\eqref{eq:highSNR} of $\Psi = \omega(\sqrt{d/n})$.
Comparing to Eqs.\eqref{eq:lowpower},\eqref{eq:highpower} to the ones above, in the very high SNR regime i.e. $\Psi = \omega (\sqrt{Bd/n})$, we have
$$
\phi_H^B = \phi_H.
$$
However, the low SNR regime is statistically more interesting. In this case, the power of the block test is $\sqrt{B}$ times worse (inside the $\Phi$ transformation). Noting that the block based test takes time $n^2/B$ to compute, we see the factor $n/\sqrt{B}$ in Eq.\eqref{eq:lowpowerB} quite illuminating (it is the square-root of the time taken). 

It was proved in \cite{powermmd2} that the power of the linear-time statistic is given by
$$
\Phi\left(\frac{\sqrt n\Psi^2}{\sqrt{8d + 8\Psi^2}} - z_\alpha \right)
$$
and hence its power in the low SNR regime is given by $\Phi\left( \frac{\sqrt n}{\sqrt {8d}} \Psi^2 \right)$ in the (very very) high SNR regime of $\Psi = \omega(\sqrt d)$, its power does not suffer, and is exactly $\Phi(\sqrt n \Psi/\sqrt 8)$ like all the above statistics, but in the low SNR regime its dependence on $n$ suffers (and again it is the square-root of the computation time taken). 

\begin{remark}
We can summarize this section informally as follows. If the test statistic takes time $n^{t}$ to compute for $1 \leq t \leq 2$ then the power behaves like $\Phi\left( \frac{n^{t/2} \Psi^2}{\sqrt {8d}} \right)$ in the low SNR regime.
\end{remark}

\section{Experiments}\label{sec:expt}

In our experience, our claimed theorems hold true much more generally in practice. For example:
\begin{enumerate}
\item While we need $n,d$ to be polynomially related in theory, we find that our experiments show that $\phi_{CQ} = \phi_{\EED} = \phi_{\GMMD}$ even when $n$ is fixed and $d$ increases, or when $d$ is fixed and $n$ increases.
\item While our theory seems to suggest that $\gamma^2 = \omega(Tr(\Sigma))$ is needed, the experiments suggest that $\gamma^2 = \Omega(Tr(\Sigma))$ suffices.
\end{enumerate}

Before we describe our experimental suite, let us first detour to mention the ``median heuristic''.

\subsection{The Median Heuristic}

The median heuristic chooses the bandwidth for the Gaussian kernel as the median pairwise distance between all pairs of points (see \cite{learningkernels}). In other words, it chooses
$$
\gamma^2 = \text{Empirical Median} \left\{ \|S - S'\|^2 \right\}
$$
where $S \neq S' \in \{X_1,...,X_n,Y_1,...,Y_n\}$.
 To have some idea of the order of magnitude of the choice that median heuristic makes, let us make the reasonable supposition that this choice  is similar to the \textit{mean}-heuristic, which chooses it to be the average distance between all pairs of points, i.e. let us assume for argument's sake that
 $$
 \text{Empirical Median} \left\{ \|S - S'\|^2 \right\} \asymp \text{Population Mean} \left\{ \|S - S'\|^2 \right\} .
 $$
Then the following proposition captures the order of magnitude of the bandwidth choice made by the common median heuristic.

\begin{proposition}\label{prop:med-heur}
Under [A1], the average distance between all pairs of points is $\asymp  2Tr(\Sigma)$. Hence, under [A1], the median-heuristic chooses $\gamma^2 \asymp 2Tr(\Sigma)$.
\end{proposition}
\begin{proof}
There are $\binom{n}{2}$ pairs of $x$s and $\binom{n}{2}$ pairs of $y$s and $n^2$ $xy$ pairs, the total number of pairs being $\binom{2n}{2}$. This implies that the population mean pairwise distance is
$\frac{\binom{n}{2}}{\binom{2n}{2}}\E \|X - X'\|^2 +\frac{\binom{n}{2}}{ \binom{2n}{2}}\E \|Y - Y'\|^2 + \frac{n^2}{\binom{2n}{2}} \E \|X-Y \|^2 $. 
\begin{eqnarray*}
\E \|X-X'\|^2 &=& \E \|(X-\mu_1) - (X'-\mu_1)\|^2 = 2 \E (X-\mu_1)^T(X-\mu_1)\\
 &=& 2 \E Tr((X-\mu_1)(X-\mu_1)^T) = 2 Tr(\Sigma).
\end{eqnarray*}
 \begin{eqnarray*}
\E \|X-Y\|^2 &=& \E \|X\|^2 + \E \|Y\|^2 - 2\E X^T Y\\
&=& \E \|X-\mu_1\|^2 + \|\mu_1\|^2 + \E \|Y - \mu_2\|^2  + \|\mu_2\|^2 - 2 \mu_1^T \mu_2\\
&=& 2Tr(\Sigma) + \|\delta\|^2.
\end{eqnarray*}
Together, these imply our claim.

\end{proof}


\begin{remark}
The above proposition implies that the choice made by the median heuristic is at the borderline of satisfying the condition under which our main theorem holds, which is $\gamma^2 = \omega(Tr(\Sigma))$. Practically, in our experiments that follow, it seems like all the claims still seem to hold even when $\gamma^2 \asymp Tr(\Sigma)$. This implies that the conditions currently needed for our theory are possibly stronger than needed. Hence, this ``heuristic'' actually provides a reasonable default bandwidth choice since $\Sigma$ is usually unknown.
\end{remark}

\subsection{Practical accuracy of our theory}

Here, we consider a wide variety of experiments and demonstrate that our claims hold true with great accuracy in practice, and actually in greater generality than we can currently prove.

The different test statistics considered in this simulation suite (as given in the legends) are:
\begin{enumerate}
\item \textbf{uMMD0.5} - $\GMMD$ with $\gamma \asymp d^{0.5}$ i.e. $\gamma^2 \asymp Tr(\Sigma)$.
\item \textbf{uMMD Median} - $\GMMD$ with $\gamma$ chosen by the aforementioned median heuristic.
\item \textbf{uMMD0.75} - $\GMMD$ with $\gamma \asymp d^{0.75}$ i.e. $\gamma^2 = \omega( Tr(\Sigma))$.
\item \textbf{ED} - (Euclidean) energy distance $\EED$, i.e $\EED_\gamma$ with $\gamma^2 = 2Tr(\Sigma)$.
\item \textbf{uCQ} - The U-statistic $U_{CQ}$ from \cite{cq}.
\item \textbf{lMMD$\#$} - The linear-time $\GMMD_L^2$ statistic from Eq.\eqref{eq:lmmd} with $\# \in \{0.5,0.75,\text{ Median}\}$ specifying the bandwidth as in the case of $\GMMD$ above.
\item \textbf{lCQ} - The linear-time version of $U_{CQ}$.
\end{enumerate}

 We plot the power of all these tests statistics when $\alpha = 0.05$, for various $P,Q$ by running 100 repetitions of the two sample test for each parameter setting. 
As a one sentence summary of all the experiments that follow, we find that all the U-statistics have exactly the same power under mean-differences, as claimed by our theorems, i.e. $\phi_{CQ} = \phi_{\GMMD} = \phi_{ED}$ for all the above choices of bandwidth, while the linear-time statistics perform significantly worse, also as predicted by the theory (demonstrating the computation-statistics tradeoff).

%

%

\paragraph{Experiment 1.} For this experiment we use the following distributions. We vary $d$ from 40 to 200 and always draw $n=d$ samples from the corresponding $P,Q$.
\begin{itemize}
\item Normal distribution with diagonal covariance: $P = N(\mu_0, I_{d \times d})$ and $Q = N(\mu_1, I_{d \times d})$ where $\mu_0 = (0 \dots 0)^\top$ and $\mu_1 = \frac{1}{\sqrt{d}} (1 \dots 1)^\top$.
\item Product of Laplace distributions: $P$ and $Q$ are shifted Laplace distributions with shifts $\mu_0 = (0 \dots 0)^\top$ and $\mu_1 = \frac{1}{\sqrt{d}} (1 \dots 1)^\top$ respectively and identity covariance matrix.
\item Product of Beta distributions: $P$ and $Q$ are shifted Beta distributions $\textsc{Beta}(1,1)$ with shifts $\mu_0 = (0 \dots 0)^\top$, $\mu_1 = \frac{1}{\sqrt{12d}} (1 \dots 1)^\top$ respectively and identity covariance matrix.
\item Mixture of Gaussian distributions: $P$ and $Q$ are shifted mixture of Gaussians $\frac{1}{3} N(0, I_{d \times d}) + \frac{1}{3} N(0, 2 I_{d \times d})+ \frac{1}{3} N(0, 3 I_{d \times d})$ with shifts $\mu_0 = (0 \dots 0)^\top$ and $\mu_1 = \sqrt{\frac{2}{d}}$ respectively.
\end{itemize}

\begin{figure} [h!]
\centering
\includegraphics[width=0.48\linewidth]{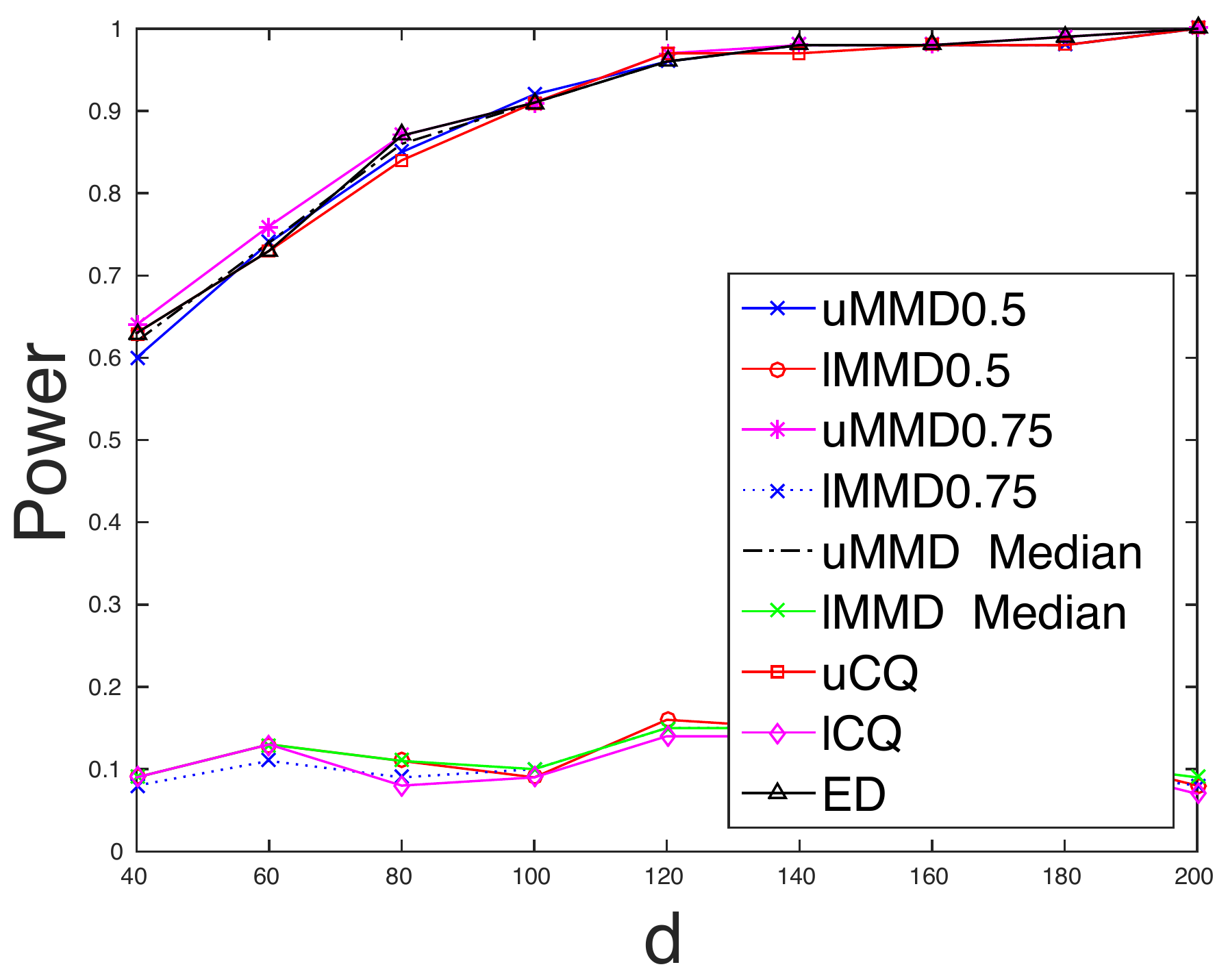}
\includegraphics[width=0.48\linewidth]{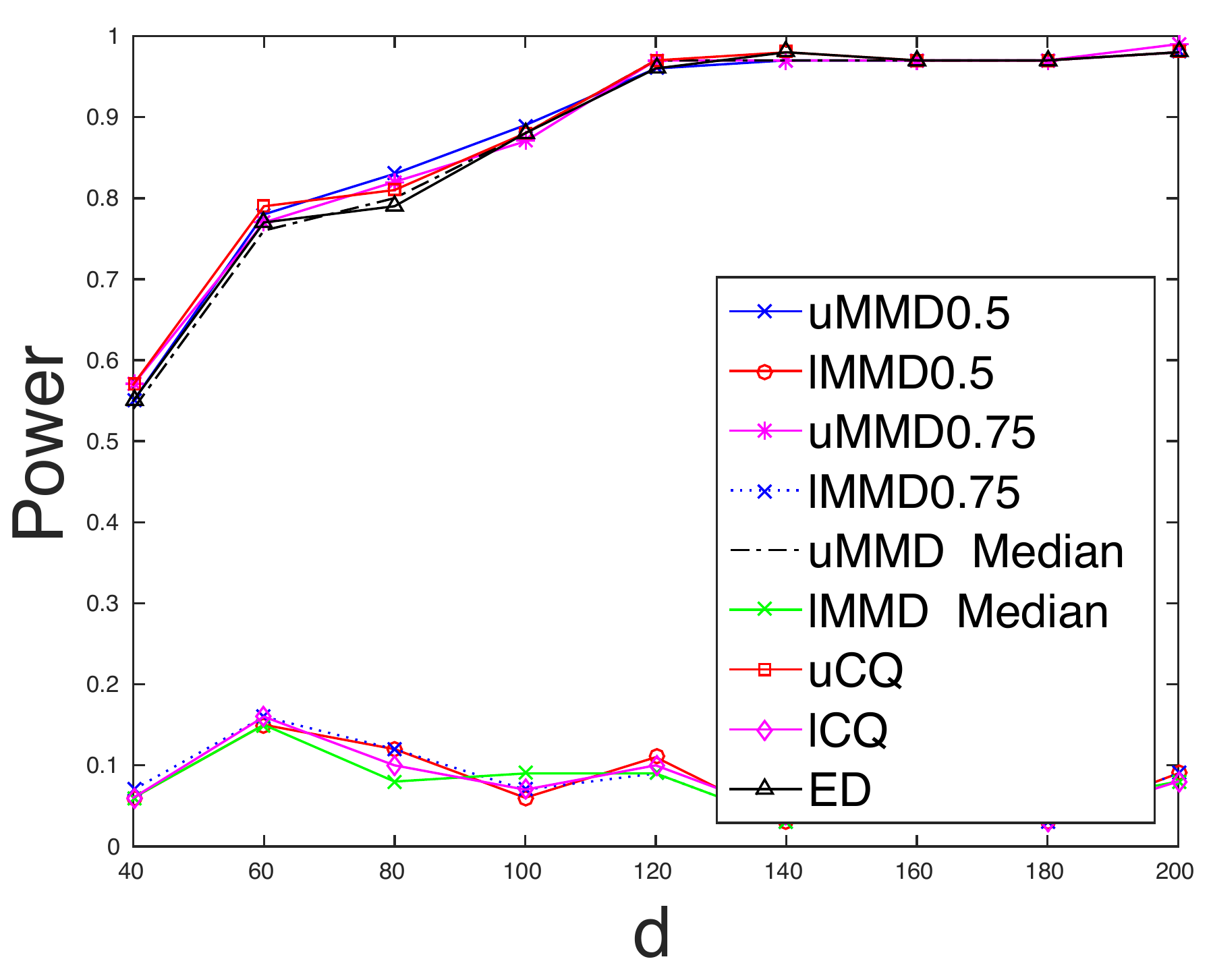}
\includegraphics[width=0.48\linewidth]{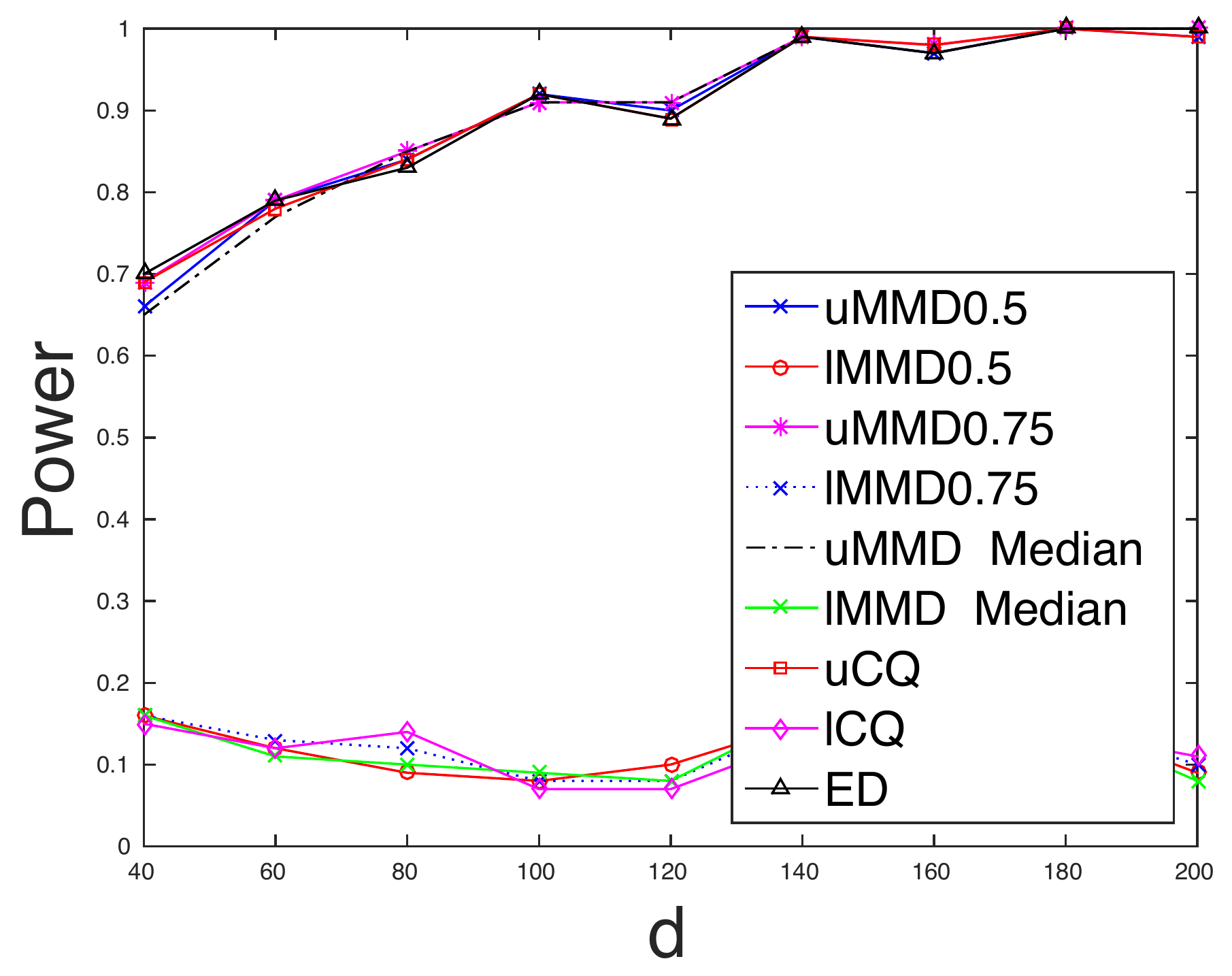}
\includegraphics[width=0.48\linewidth]{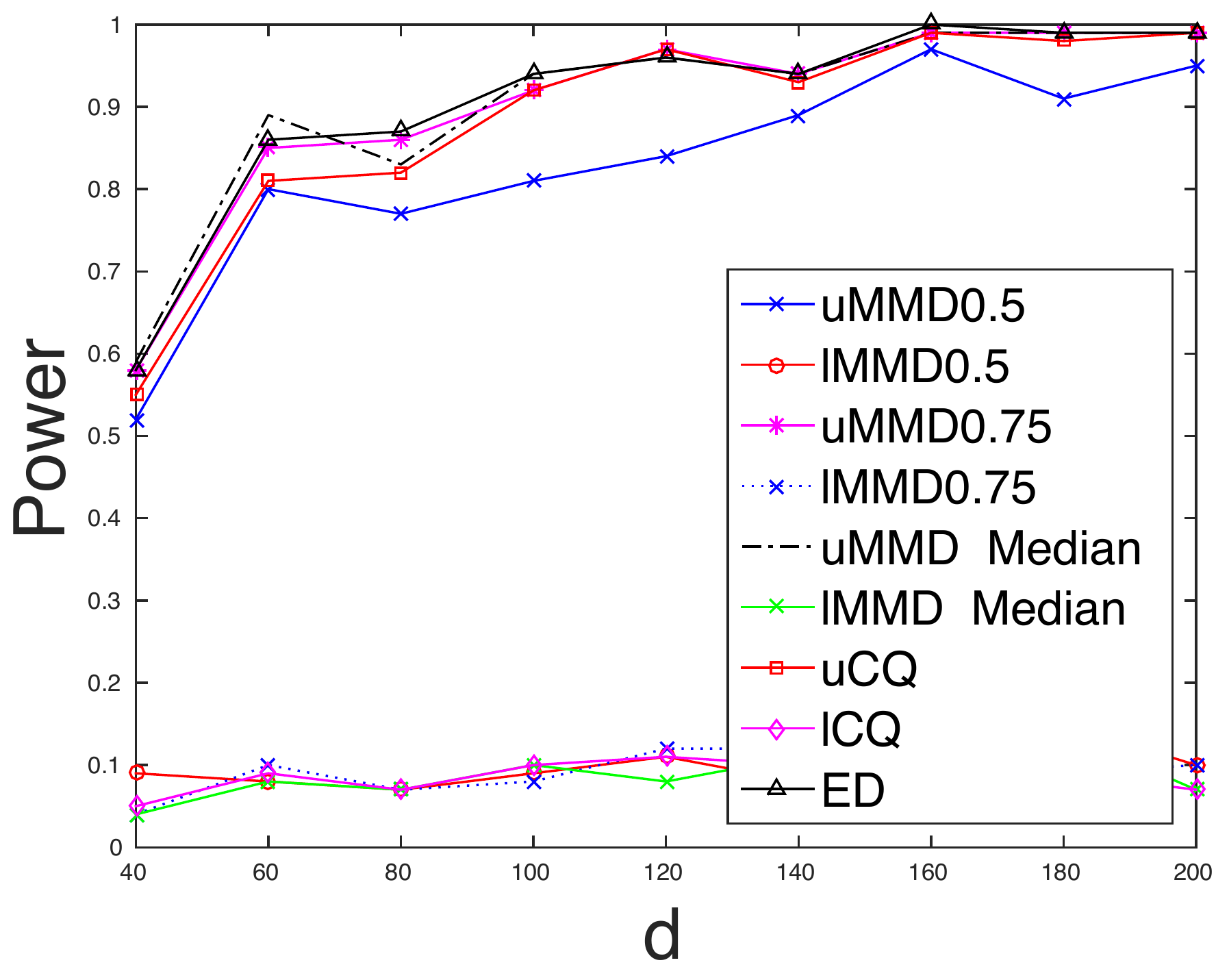}
\caption{Power vs Dimension when $P,Q$ are mean-shifted Normal (top left), Laplace (top right), Betas (bottom left) or Mixture (bottom right plot) distributions.}
\label{fig:exp1}
\end{figure}

The values of shifts and covariance matrix are chosen to keep the asymptotic power same for all the distribution (see Theorem~\ref{thm:gmmd}). Figure~\ref{fig:exp1} shows the performance of various estimators for the aforementioned two sample test settings. It is clear that the power of $\EED, T_{CQ},\GMMD$ all coincide for any (sufficiently large) bandwidth, increasing as $\Phi(\sqrt n)$ for the quadratic time statistic, and staying constant for the linear time statistics, both as predicted by the theory. Also note the fact that the plots look almost identical is consistent with our theory (see Theorem~\ref{thm:gmmd}).

{\bf Experiment 2}: In the previous experiment, we have seen the performance of the estimators for diagonal covariance matrix. Here, we empirically verify that similar effects can be observed in distributions with non-diagonal covariance matrix. To this end, we consider distributions $P = N(\mu_0, \Sigma')$ and $Q = N(\mu_1, \Sigma')$ where $\mu_0 = (0 \dots 0)^\top$, $\mu_1 = \frac{1}{\sqrt{d}}(1 \dots 1)^\top$ and $\Sigma' = U \Lambda' U^\top$. The matrix $U$ is a random unitary matrix $U$ obtained from the eigenvectors of a random Gaussian matrix. $\Lambda'$ is set as follows. Let $\Lambda$ be a diagonal matrix, the entries of which are equally spaced between 0.01 and 1, raised to the power 6. This experimental setup is similar to one used in \cite{lopes11}. The matrix $\Lambda'$ is $d\frac{\Lambda}{tr(\Lambda)}$. Figure~\ref{fig:exp2} shows that the qualitative performance of all statistics is similar to one observed in the previous experiment (see Figure~\ref{fig:exp1}).

\begin{figure} [h!]
\centering
\includegraphics[width=0.5\linewidth]{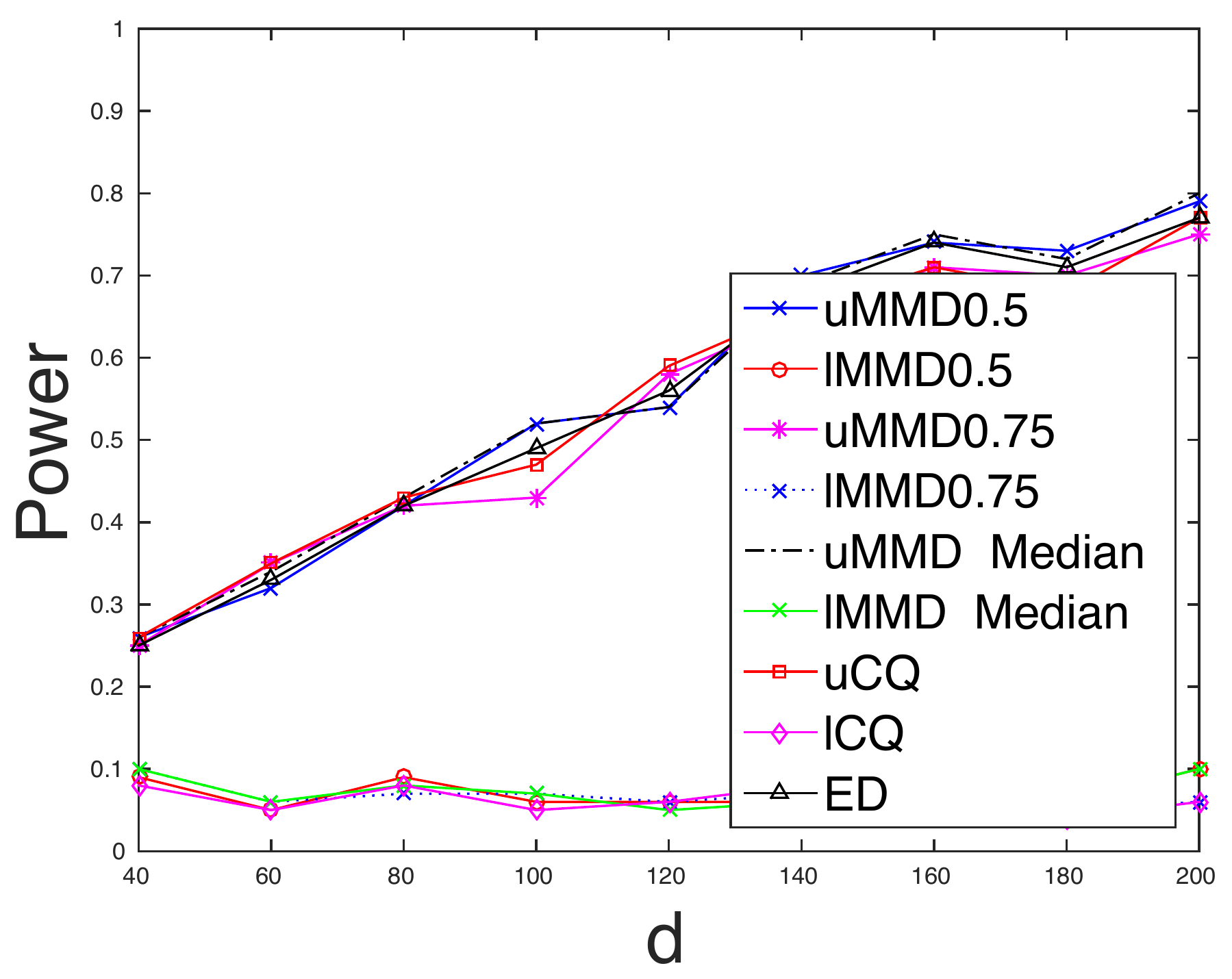}
\caption{Power vs $d$ when $P,Q$ are mean-shifted Normal (top left) with non-diagonal covariance matrix.}
\label{fig:exp2}
\end{figure}

%

\paragraph{Experiment 4.} The aim of this experiment is to study the performance of the statistics when distributions differ in covariances rather than means. In this experiment, we set $P = N(0,\Sigma_1)$ and $Q = N(0,\Sigma_2)$ where $\Sigma_1 = \frac{50 I}{\|\Sigma\|_F}$ and $\Sigma_2 = \frac{50 (\Sigma + I)}{\|\Sigma\|_F}$. Here, $\Sigma$ is a positive definite matrix $U \Lambda U^\top$ where $U$ and $\Lambda$ are generated as described in Experiment 2. Again, the experimental setup is similar to the one used in \cite{lopes11}. Not surprisingly, as seen in Figure~\ref{fig:exp4}, $\GMMD$ and $\EED$ perform better than CQ.

\begin{figure} [h!]
\centering
\includegraphics[width=0.45\linewidth]{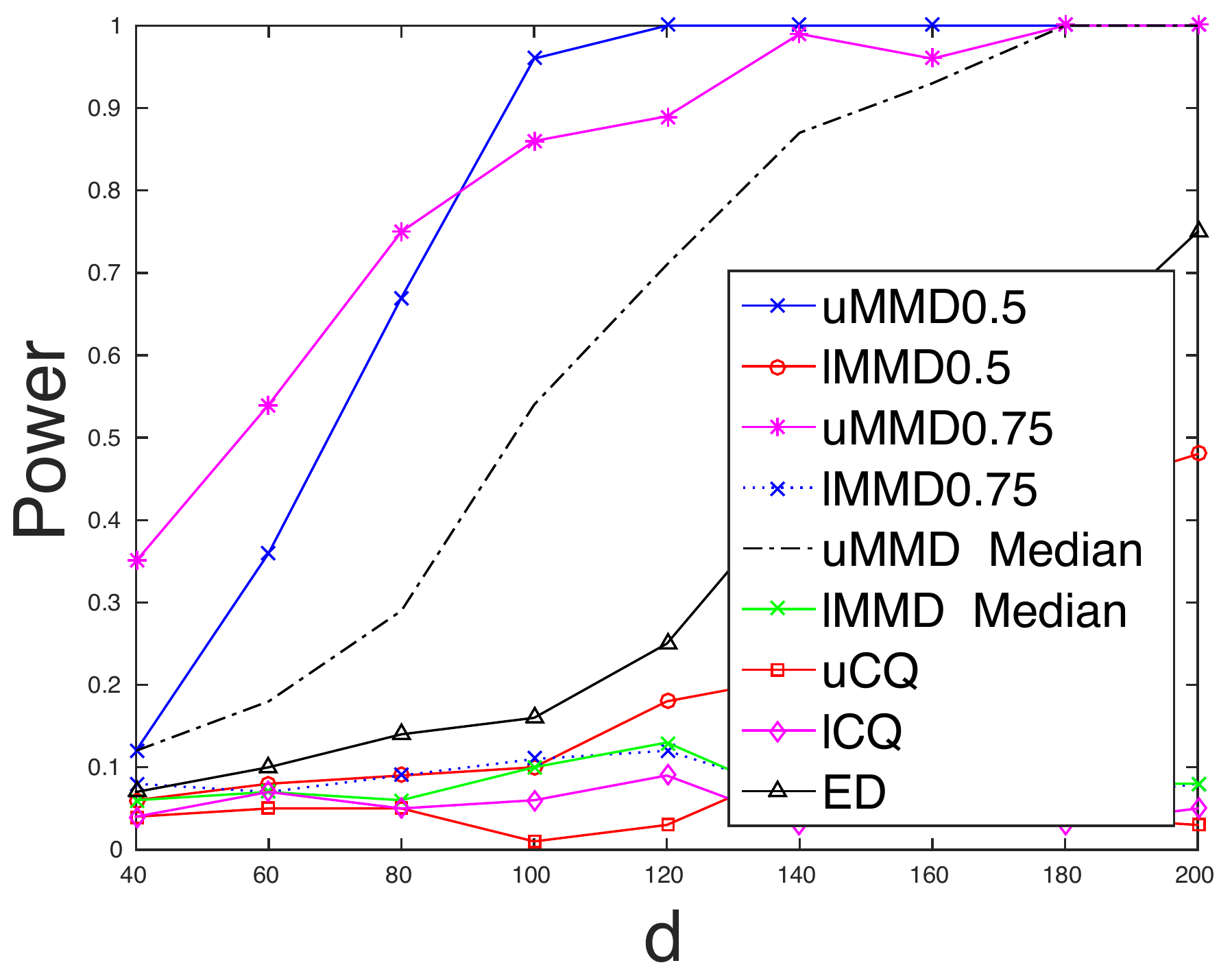}
\caption{Power vs $d$ when $P,Q$ are distributions differing in Covariances.}
\label{fig:exp4}
\end{figure}

 This experiment demonstrates that $\GMMD$ and $\EED$ \textit{dominate} $U_{CQ}$ in some sense. This is due to the fact that CQ is designed for mean-shift alternatives while rest of them work for more general alternatives. Hence, they achieve the same power when the distributions differ in their means, and strictly higher power when the distributions do not differ in their means, but only in some higher moment. We can also see that the powers of the different statistics are no longer equal, and that the bandwidth does matter in this situation.

\paragraph{Experiment 5.} Finally, we verify the nature of the asymptotic power for fixed dimension. For the purpose of this experiment, we hold $d$ fixed to value 40 and vary $n$. Here, we consider two sample tests for normal distributions with diagonal and non-diagonal covariance matrices (used in Experiment 1 and Experiment 2 respectively). Figure~\ref{fig:exp5} illustrates the power of the tests under this scenario. It can be seen that power increases with $n$ in a manner similar to the ones observed in the previous experiments.

\begin{figure} [h!]
\centering
\includegraphics[width=0.45\linewidth]{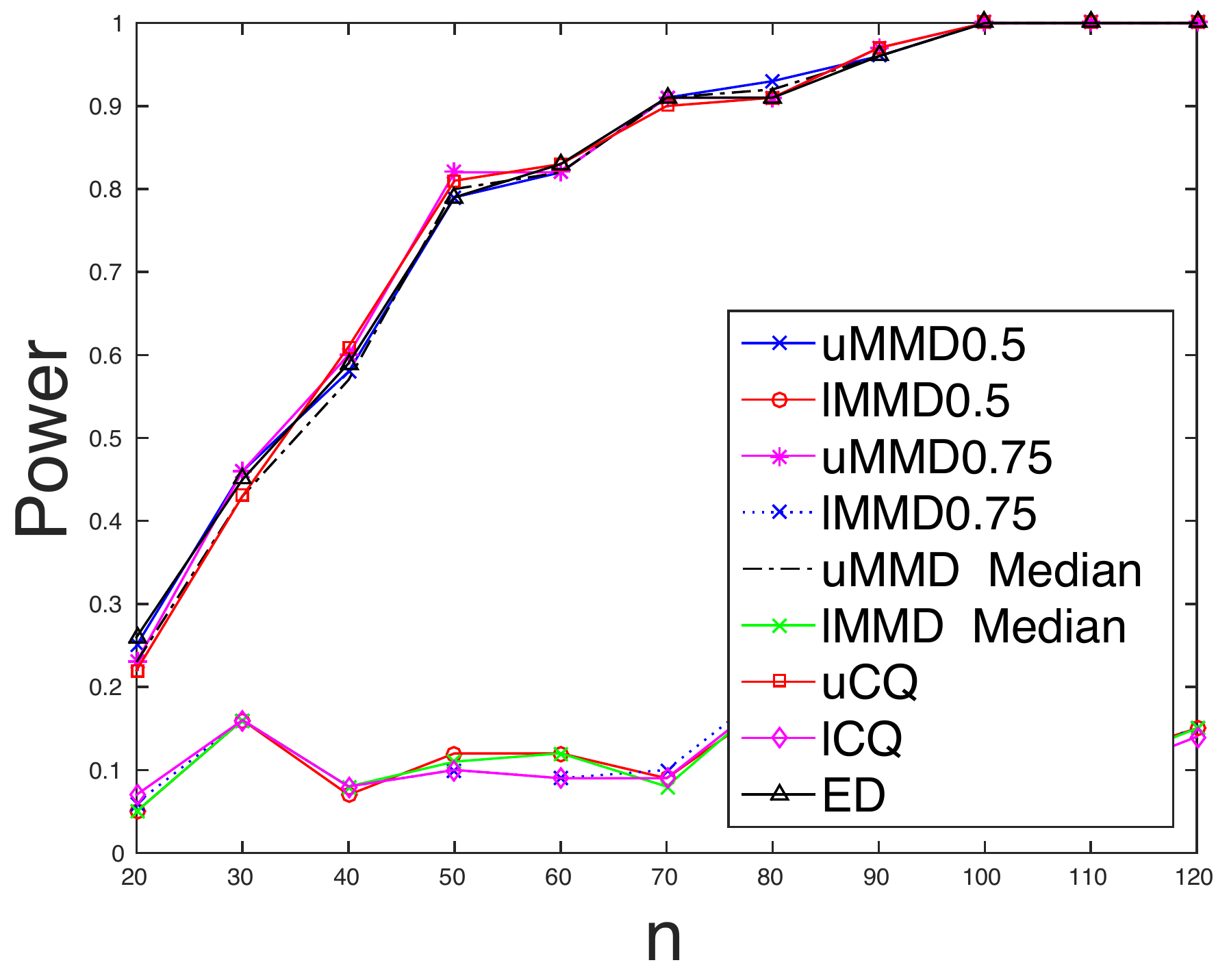}
\includegraphics[width=0.45\linewidth]{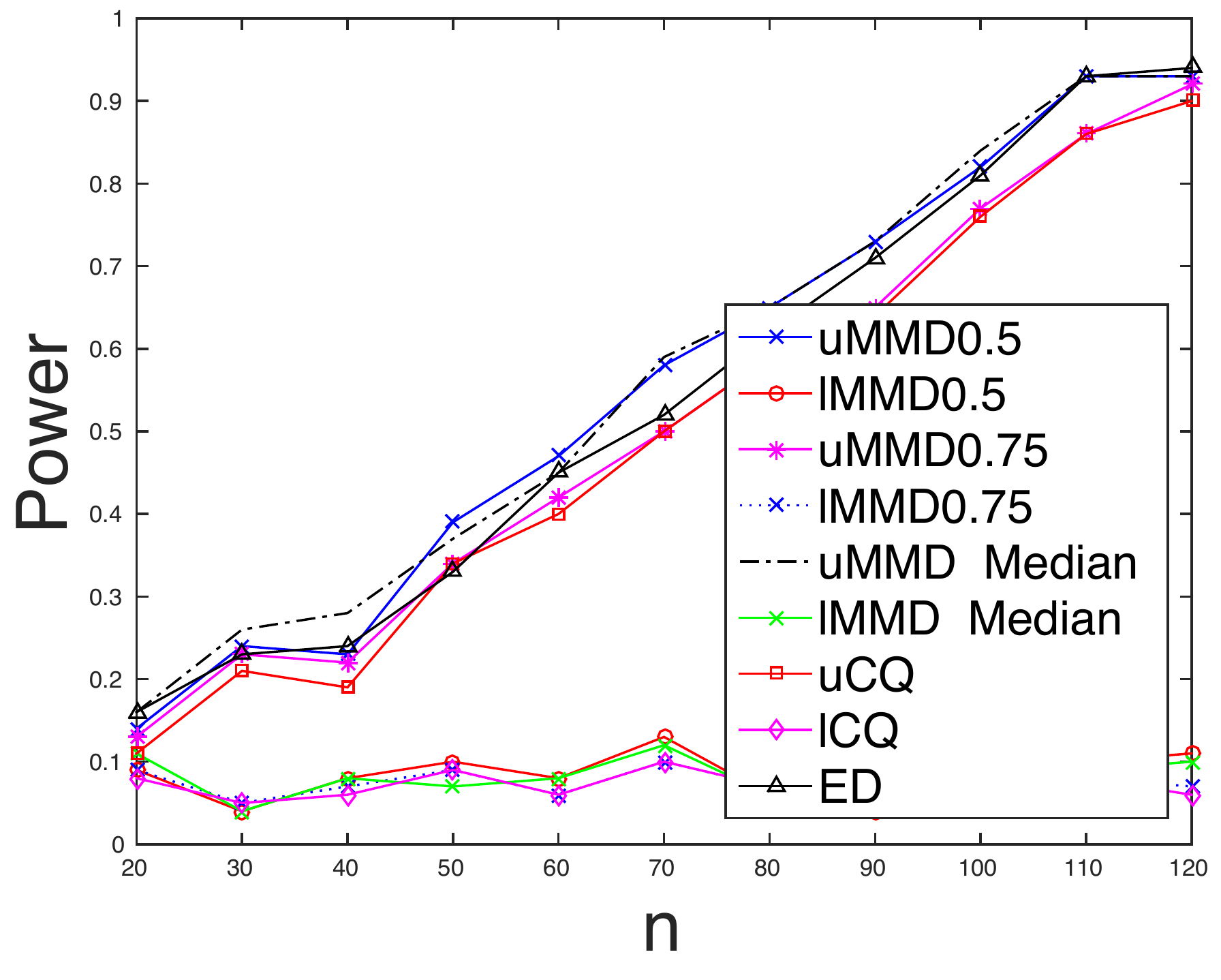}
\caption{Power vs Sample size for fixed dimension when $P,Q$ are normal distributions with diagonal (left plot) and non-diagonal (right plot) covariance matrices respectively.}
\label{fig:exp5}
\end{figure}

This experiment suggests that assumption [A5] can probably be relaxed or dropped from the theory. We need it only to bound a certain Taylor remainder term $R_3$ in the proof of the theorems that follows, and it is perhaps possible to find a better way to bound this term.

%


\section{Proofs of Theorems \ref{thm:gmmd} and \ref{thm:eed}}\label{sec:proofs}

Let us first note that the $\GMMD$ statistic can be  written as 

\begin{eqnarray}
\GMMD &=& 
\begin{bmatrix}
\textbf{1}_n/\sqrt{n(n-1)}\\
-\textbf{1}_n/\sqrt{n(n-1)}
\end{bmatrix}^T
 \begin{bmatrix}
K_{XX} & K_{XY} \\
K_{XY}^T & K_{YY}
\end{bmatrix}
\begin{bmatrix}
\textbf{1}_n/\sqrt{n(n-1)}\\
-\textbf{1}_n/\sqrt{n(n-1)}
\end{bmatrix}\nonumber \\
&=& 
\frac{2}{(n-1)} \cdot 
u^T
K
u \label{eq:uTKu}
\end{eqnarray}

where $ u = \begin{bmatrix}
\textbf{1}_n/\sqrt{2n}\\
-\textbf{1}_n/\sqrt{2n}
\end{bmatrix} $ is a unit vector and $K =  \begin{bmatrix}
K_{XX} & K_{XY} \\
K_{YX} & K_{YY}
\end{bmatrix}$
with its submatrices defined as

\begin{eqnarray*}
K_{XX } &:=& 
\left\{
\exp\left(-\frac{\|X_i - X_j\|^2}{\gamma^2}  \right) \mathbb{I}(i \neq j) 
\right\}\\
&:=&
\begin{bmatrix}
0 & \exp\left(-\frac{\|X_1 - X_2\|^2}{\gamma^2} \right) & \cdots & \exp\left(-\frac{\|X_1 - X_n\|^2}{\gamma^2}\right)\\
\exp\left(-\frac{\|X_2 - X_1\|^2}{\gamma^2}\right) & 0 & \cdots & \exp\left(-\frac{\|X_2 - X_n\|^2}{\gamma^2}\right)\\
\vdots & \vdots  & \ddots & \vdots\\
\exp\left(-\frac{\|X_n - X_1\|^2}{\gamma^2}\right) & \exp\left(-\frac{\|X_n - X_2\|^2}{\gamma^2}\right) & \cdots & 0
\end{bmatrix}
\end{eqnarray*}

and we use the first expression to summarize the above matrix and similarly,

\[
K_{YY } 
= 
\left\{
\exp\left(-\frac{\|Y_i - Y_j\|^2}{\gamma^2}\right) \mathbb{I}(i \neq j) 
\right\}
\]

\[
K_{XY } = K_{YX}^T = 
\left\{
\exp\left(-\frac{\|X_i - Y_j\|^2}{\gamma^2}\right) \mathbb{I}(i \neq j) 
\right\}
\]

Note that there are 0s on the diagonal of $K$, but also on the diagonals of the other two submatrices.
Note that $2Tr(\Sigma) + \|\delta\|^2 = \E\|X_i - Y_j \|^2 \asymp \E\|X_i - X_j\|^2 = \E\|Y_i - Y_j \|^2 = 2Tr(\Sigma)$ since $\|\delta\|^2 = o(Tr(\Sigma))$ by Assumption [A4]. For $i \neq j$, let 
\begin{equation}\label{eq:tau}
\tau := 2Tr(\Sigma)/\gamma^2  \asymp  \E\|S_i - S_j\|^2/\gamma^2 = o(1)
\end{equation}
 for $S_i \in \{X_i,Y_i\}$. Let $a=\frac{\|S_i - S_j\|^2}{\gamma^2}$ Let us write the exact third order Taylor expansion of the terms $\exp(-a)$ around $\exp(-\tau)$ as
\begin{eqnarray}
e^{-a} &=& e^{-\tau} - e^{-\tau} ( a - \tau ) 
 + \frac{e^{-\tau}}{2} \left( a - \tau \right)^2  - \frac{e^{-\zeta_{ij}}}{3!} \left( a - \tau \right)^3 \label{eq:gmmdtaylor}
\end{eqnarray}
for some $\zeta_{ij}$ between $a$ and $\tau$, and since $a,\tau > 0$, we have
$\exp(-\zeta_{ij}) \leq 1$. For clarity in the following expressions, we drop the $\mathbb{I}(i \neq j)$ and assume it is understood. In this notation, the term-wise Taylor expansion of $K$ is given by
\begin{eqnarray*}
K &=& \begin{bmatrix}
\left\{ e^{-\frac{\|X_i - X_j\|^2}{\gamma^2}} \right\} & 
\left\{ e^{-\frac{\|X_i - Y_j\|^2}{\gamma^2}} \right\}
\\
\left\{ e^{-\frac{\|Y_i - X_j\|^2}{\gamma^2}} \right\} & 
\left\{ e^{-\frac{\|Y_i - Y_j\|^2}{\gamma^2}} \right\}
\end{bmatrix}\\
&=& 
e^{-\tau}  \begin{bmatrix}
\left\{ 1 \right\} & \left\{ 1 \right\}
\\
\left\{ 1 \right\} & \left\{ 1 \right\}
\end{bmatrix}
- 
e^{-\tau}
\begin{bmatrix}
\left\{ \frac{\|X_i - X_j\|^2}{\gamma^2} - \tau \right\} & \left\{  \frac{\|X_i - Y_j\|^2}{\gamma^2} - \tau \right\}
\\
\left\{  \frac{\|Y_i - X_j\|^2}{\gamma^2} - \tau \right\} & \left\{  \frac{\|Y_i - Y_j\|^2}{\gamma^2} - \tau \right\}
\end{bmatrix}\\
&& + \frac{e^{-\tau}}{2!}
\begin{bmatrix}
\left\{ \left( \frac{\|X_i - X_j\|^2}{\gamma^2} - \tau \right)^2 \right\} & 
\left\{  \left( \frac{\|X_i - Y_j\|^2}{\gamma^2} - \tau \right)^2 \right\}
\\
\left\{ \left( \frac{\|Y_i - X_j\|^2}{\gamma^2} - \tau \right)^2  \right\} &
 \left\{ \left( \frac{\|Y_i - Y_j\|^2}{\gamma^2} - \tau \right)^2  \right\}
\end{bmatrix}\\
&& - 
\frac1{3!}
\begin{bmatrix} 
\left\{ e^{-\zeta^{XX}_{ij}} \left( \frac{\|X_i - X_j\|^2}{\gamma^2} - \tau \right)^3 \right\} & 
\left\{  e^{-\zeta^{XY}_{ij}} \left( \frac{\|X_i - Y_j\|^2}{\gamma^2} - \tau \right)^3 \right\}
\\
\left\{ e^{-\zeta^{YX}_{ij}} \left( \frac{\|Y_i - X_j\|^2}{\gamma^2} - \tau \right)^3 \right\} &
 \left\{ e^{-\zeta^{YY}_{ij}} \left( \frac{\|Y_i - Y_j\|^2}{\gamma^2} - \tau \right)^3  \right\}
\end{bmatrix}
\end{eqnarray*}

%
%

Recalling Eq.\eqref{eq:uTKu} and expanding using the above Taylor expansion of $K$, we get
\begin{equation}\label{eq:gmmdcq}
\GMMD = 2e^{-\tau} \frac{U_{CQ} }{\gamma^2} + \frac{e^{-\tau}}{(n-1)} u^T T_2 u - \frac2{3! (n-1)} u^T (E \circ T_3) u
\end{equation}
where, recalling that $\circ$ is the Hadamard product,
\begin{eqnarray*}
T_2 &:=& \begin{bmatrix}
\left\{ \left( \frac{\|X_i - X_j\|^2}{\gamma^2} - \tau \right)^2 \right\} & 
\left\{  \left( \frac{\|X_i - Y_j\|^2}{\gamma^2} - \tau \right)^2 \right\}
\\
\left\{ \left( \frac{\|Y_i - X_j\|^2}{\gamma^2} - \tau \right)^2  \right\} &
 \left\{ \left( \frac{\|Y_i - Y_j\|^2}{\gamma^2} - \tau \right)^2  \right\}
\end{bmatrix}\\
E &:=& \begin{bmatrix}
\{ e^{-\zeta^{XX}_{ij}} \} & \{ e^{-\zeta^{XY}_{ij}} \}\\
\{ e^{-\zeta^{YX}_{ij}} \} & \{ e^{-\zeta^{YY}_{ij}} \}
\end{bmatrix}\\
T_3 &:=& \begin{bmatrix}
\left\{ \left( \frac{\|X_i - X_j\|^2}{\gamma^2} - \tau \right)^3 \right\} & 
\left\{  \left( \frac{\|X_i - Y_j\|^2}{\gamma^2} - \tau \right)^3 \right\}
\\
\left\{ \left( \frac{\|Y_i - X_j\|^2}{\gamma^2} - \tau \right)^3  \right\} &
 \left\{ \left( \frac{\|Y_i - Y_j\|^2}{\gamma^2} - \tau \right)^3  \right\}
\end{bmatrix}.
\end{eqnarray*}

Note that we have used the fact that for $ u = \begin{bmatrix}
\textbf{1}_n/\sqrt{2n}\\
-\textbf{1}_n/\sqrt{2n}
\end{bmatrix} $ we have
$$
u^T \begin{bmatrix}
\left\{ 1 \right\} & \left\{ 1 \right\}
\\
\left\{ 1 \right\} & \left\{ 1 \right\}
\end{bmatrix} u = 0
$$
and also that
$$
U_{CQ} =  \frac1{\binom{n}{2}} \sum_{i \neq j} \left\{- \|X_i - X_j\|^2   - \|Y_i - Y_j\|^2 + \|X_i - Y_j\|^2 + \|X_j - Y_i\|^2  \right\}.
$$
Further, recall from Eq.\eqref{eq:tau} that $\tau = o(1)$.

The proof of the theorem will proceed from Eq.\eqref{eq:gmmdcq} in three steps. Define
\begin{eqnarray*}
U_4 &:=& \frac1{\binom{n}{2}} \sum_{i \neq j} h_4(X_i,X_j,Y_i,Y_j) \\
h_4(X_i,X_j,Y_i,Y_j) &:=& \|X_i-X_j\|^4 + \|Y_i - Y_j\|^4 - \|X_i - Y_j\|^4 - \|X_j - Y_i\|^4
\end{eqnarray*}
to note that 
$$
\frac1{(n-1)}u^T T_2 u = \left( \frac{U_4}{2\gamma^4} +  \frac{\tau U_{CQ}}{\gamma^2} \right)
$$
\begin{enumerate}
\item[(i)] First we will show that the third order Taylor remainder term $R_3 := \frac2{3! (n-1)} u^T (E \circ T_3) u$ is a smaller order term than $U_{CQ}/\gamma^2$. 

\item[(ii)] Denote $\theta_2 = \frac1{n-1}u^T \E[T_2] u$. We will show that $\theta_2 = o(\|\delta\|^2/\gamma^2)$.

\item[(iii)]  Denote $s_4 = Var(U_4)$. We will show that $Var(U_4/\gamma^4) = o(Var(U_{CQ}/\gamma^2))$.
\end{enumerate}

 Both $\theta_4$ and $s_4$ are  tedious to calculate, especially under the alternative, and we will have to develop a series of lemmas on the way to calculate these quantities.
Assuming for the moment that these above claims are true, we then have from Eq.\eqref{eq:gmmdcq} that
$$
\GMMD =  \frac{U_{CQ}}{\gamma^2} (2e^{-\tau} + o_P(1))
$$
Since we have assumed $ m\geq 8$ moments, this immediately implies convergence of means and variances, i.e.
\begin{equation}\label{eq:Egmmd}
\E \GMMD = \frac{\|\delta\|^2}{\gamma^2} (2e^{-\tau} + o(1))
\end{equation}
and
\begin{equation}\label{eq:Vgmmd}
Var(\GMMD) = \frac{Var(U_{CQ})}{\gamma^4} (2e^{-\tau} + o(1))^2
\end{equation}
which then implies that, ignoring smaller order terms,
$$
\frac{\GMMD - \E \GMMD}{\sqrt{Var(\GMMD)}} = \frac{U_{CQ} - \|\delta\|^2}{\sqrt{8\frac{Tr(\Sigma^2)}{n^2} + 8 \frac{\delta^T \Sigma \delta }{n}}} 
$$
and hence the distribution of $\GMMD$ matches the distribution of $U_{CQ}$ under null and alternative (and the above expression has a standard normal distribution), and the two statistics hence also have the same power. The same argument also holds for the studentized statistics calculated in practice. The rest of the proof is devoted to proving the three steps (i), (ii) and (iii).

\subsection*{Step (i): Bounding $R_3:= \frac2{3! (n-1)} u^T (E \circ T_3) u$}
Noting that every element of $E$ is smaller than 1, and hence $u^T(E \circ T_3) u \leq \|E \circ T_3\|_2 \leq \max_{ij} E_{ij} \|T_3\|_2 \leq \|T_3\|_2$, implying that (ignoring constants)
$$
R_3 \leq \frac{\|T_3\|_2}{n} \leq \frac{\|T_3\|_\infty}{\sqrt{n}}
$$
Let us now bound every term of $T_3$.
Taking a union bound on the statement of Assumption [A3], we see that the same exponential concentration bound holds uniformly for all $O(n^2) = o(d^4)$ pairs $i,j$, and hence w.p. tending to 1,
$$
\max_{ij} \Bigg\vert  \frac{\|S_i - S_j\|^2}{\gamma^2} - \tau \Bigg\vert  \leq d^{-\nu(\Sigma,m)} \frac{d}{\gamma^2}
$$
 (we also multiplied both sides by $d/\gamma^2$). Hence we have w.p. tending to 1, 
 $$
 R_3 \leq \frac{1}{d^{3\nu}\sqrt n} \frac{d^3}{\gamma^6}
 $$
Since any random variable satisfies $X = O_P(\sqrt{Var(X)}) $, we have that $U_{CQ}/\gamma^2 = O_P\left( \frac{\sqrt{Tr(\Sigma^2)}}{n \gamma^2} \right)$ under the null (its variance is even larger under the alternate), and hence 
$
R_3 = o_P\left( U_{CQ}/\gamma^2 \right)
$
whenever
$$
\frac{1}{d^{3\nu}\sqrt n} \frac{d^3}{\gamma^6} = o \left(\frac{\sqrt{Tr(\Sigma^2)}}{n \gamma^2} \right) \text{~ i.e. ~} \sqrt n  = o\left( \frac{\gamma^4 \sqrt{Tr(\Sigma^2)}}{d^{3-3\nu}} \right)
$$
This is reasonably satisfied whenever $\gamma^2 > Tr(\Sigma) \asymp d$ and $n = o(d^{3\nu - 1}Tr(\Sigma^2))$ as assumed. Hence, under our assumptions $R_3 = o_P(U_{CQ}/\gamma^2)$.  

\textbf{Remark.} We conjecture that this holds true under much weaker conditions on $\gamma, n, \Sigma, m$.

\subsection*{Step (ii): The Behavior of $\theta_4 = \E[U_4]$ and $\theta_2 = \frac1{n-1}u^T \E[T_2] u$}

Note the fact that for any random variable $V$, $\E(V - b)^2 = Var(V) + (\E V - b)^2$. Using $V = \|X-Y\|^2/\gamma^2$, $b=\tau$ and $\E V = \tau + \|\delta\|^2/\gamma^2$, we can write the off-diagonal terms as
\begin{eqnarray*}
\E
\begin{bmatrix}
\left\{ \left( \frac{\|X_i - X_j\|^2}{\gamma^2} - \tau \right)^2 \right\} & 
\left\{  \left( \frac{\|X_i - Y_j\|^2}{\gamma^2} - \tau \right)^2 \right\}
\\
\left\{ \left( \frac{\|Y_i - X_j\|^2}{\gamma^2} - \tau \right)^2  \right\} &
 \left\{ \left( \frac{\|Y_i - Y_j\|^2}{\gamma^2} - \tau \right)^2  \right\}
\end{bmatrix} &=& 
\begin{bmatrix}
\left\{ \frac{Var(\|X-X'\|^2)}{\gamma^4} \right\} & 
\left\{ \frac{Var(\|X-Y\|^2)}{\gamma^4} + \frac{\|\delta\|^4}{\gamma^4}  \right\}
\\
\left\{ \frac{Var(\|X-Y\|^2)}{\gamma^4} + \frac{\|\delta\|^4}{\gamma^4} \right\} &
 \left\{ \frac{Var(\|Y-Y'\|^2)}{\gamma^4} \right\}
\end{bmatrix}
\end{eqnarray*}
Since $Var(\|X-X'\|^2) = Var(\|Y-Y'\|^2)$, we have
$$
\theta_2 = Var(\|X-X'\|^2) - Var(\|X-Y\|^2) - \|\delta\|^4/\gamma^4.
$$
The next two propositions imply that
$\theta_2 = -8 \delta^T \Sigma \delta/\gamma^4 - \|\delta\|^4/\gamma^4 = o(\|\delta\|^2/\gamma^2)$, as required for step (ii).
They also imply that
$$
\theta_4 = -16 \delta^T \Sigma \delta - 8\|\delta\|^2 Tr(\Sigma) - 2 \|\delta\|^4 \asymp -\|\delta\|^2 Tr(\Sigma).
$$



\begin{proposition}\label{prop:zSz}
Define $Z' = Z_1-Z_2$ where $Z_1, Z_2$ are as in assumption [A1], [A2]. Then
\begin{eqnarray*}
\E(Z'^T \Sigma Z') &=&  2 Tr(\Sigma) \label{eq:varzSz}\\
Var(Z'^T \Sigma Z') &\asymp&  Tr(\Sigma^2) \label{eq:varzSz}\\
\E[(Z'^T \Sigma Z')^2] &\asymp& Tr^2(\Sigma)
\end{eqnarray*}
\end{proposition}
\begin{proof}
Since $Z_1, Z_2$ are independent, zero mean and identity covariance, we have $Z'$ is mean zero and covariance $2I$ and fourth moment $\E Z'^4_k = \E (Z_{1k} - Z_{2k})^4 = 3 + \Delta_4+ 6  + 3 + \Delta_4= 12 + 2 \Delta_4$. Firstly
\begin{eqnarray*}
\E [Z'^T \Sigma Z'] &=& \E Tr(Z'^T \Sigma Z') 
= Tr \E (Z'^T \Sigma Z')
= Tr ( \E (\Sigma Z'Z'^T))\\
&=& 2Tr(\Sigma)
\end{eqnarray*}
where the last step follows since $\E[Z'Z'^T] = 2I$.
\begin{eqnarray*}
Var(Z'^T \Sigma Z') &=& \E [Z'^T \Sigma Z']^2 - [2Tr(\Sigma)]^2 = \E \sum_{i,j,k,l} \Sigma_{ij}\Sigma_{kl}Z'_i Z'_j Z'_k Z'_l - 4(\sum_i \Sigma_{ii})^2 \\
&=& 4 \sum_{i }\sum_{j \neq i} \Sigma_{ii} \Sigma_{jj} + 8 \sum_{i} \sum_{j \neq i} \Sigma_{ij}^2  + (12+4\Delta_4) \sum_i \Sigma_{ii}^2 - 4 (\sum_i \Sigma^2_{ii} + \sum_{i}\sum_{j \neq i} \Sigma_{ii}\Sigma_{jj}) \\
&=& 8 Tr(\Sigma^2) + 4\Delta_4 Tr(\Sigma \circ \Sigma)
\end{eqnarray*}
where the third step follows because the only nonzero terms in $\sum_{i,j,k,l}$ are because (a) $i=j$ and $k=l \neq i$ or (b) $i=k$ and $j=l \neq i$ or (c) $i=l$ and $j=k \neq i$ or (d) $i=j=k=l$ and the last step follows because $Tr(\Sigma^2) = \|\Sigma\|_F^2 = \sum_{i,j} \Sigma_{ij}^2 $. The lemma is proved because $\sum_i \Sigma_{ii}^2 \leq \sum_{i,j} \Sigma^2_{ij}$.
\begin{eqnarray*}
\text{Hence } \E[(Z'^T \Sigma Z')^2] &=& Var(Z'^T \Sigma Z') + (\E Z'^T \Sigma Z')^2 =  8 Tr(\Sigma^2) + 2\Delta_4\sum_i \Sigma_{ii}^2 + 4Tr^2(\Sigma)\\
&\asymp& Tr^2(\Sigma).
\end{eqnarray*}
\end{proof}


\begin{proposition}\label{prop:x-y}
Let $X,Y$ be as in assumption [A1], [A2], [A3].  Then
\begin{eqnarray*}
\E\|X-Y\|^2 &=& 2Tr(\Sigma) + \|\delta\|^2, \label{eq:expx-y}\\
Var(\|X -Y\|^2)  &\asymp& 8 Tr(\Sigma^2) + 8\delta^T \Sigma \delta, \label{eq:varx-y}\\
\E \|X - Y\|^4 &\asymp& 4 Tr^2(\Sigma) + 4 \|\delta\|^2Tr(\Sigma),\\
\end{eqnarray*}
\end{proposition}

\begin{proof}
Remember that $X - Y = \Gamma (Z_1 - Z_2) + \delta =: \Gamma Z' + \delta$. Note that $Z'$ has zero mean, variance $2I$ and every component is independent with third moment zero. Hence
\begin{eqnarray*}
\E\|X-Y\|^2 &=& \E \|\Gamma Z' + \delta\|^2 = \E [Z'^T\Pi Z'] + \|\delta\|^2 + 2\E[\delta^T \Gamma Z']\\
&=& 2Tr(\Sigma) + \|\delta\|^2.\\
\text{Hence } Var\|X-Y\|^2 &=& \E[ \|\Gamma Z' + \delta\|^2 - (2Tr(\Sigma) + \|\delta\|^2)]^2\\
&=& \E [Z'^T \Pi Z' + 2\delta^T \Gamma Z' - 2Tr(\Sigma)]^2\\
&=& Var(Z'^T \Pi Z') + 4\E[\delta^T \Gamma Z' Z'^T \Gamma^T \delta] + 4\E[(Z'^T \Pi Z' - 2Tr(\Sigma))\delta^T \Gamma Z']\\
&=& 8Tr(\Sigma^2) + 4\Delta_4 Tr(\Sigma \circ \Sigma) + 8 \delta^T \Sigma \delta + 4 \E\left[ \sum_{i,j}\Pi_{ij} Z'_i Z'_j Z'^T\right]\Gamma^T \delta\\
&=& 8Tr(\Sigma^2) + 4\Delta_4 Tr(\Sigma \circ \Sigma) + 8 \delta^T \Sigma \delta 
\end{eqnarray*}
The second last step follows since $\E \sum_{i,j}\Pi_{ij} Z'_i Z'_j Z_k'= 0$ since $Z'$ has first and third moments 0.

\begin{eqnarray*}
\text{Hence } \E \|X-Y\|^4 &=& Var (\|X-Y\|^2) + (\E \|X-Y\|^2)^2\\
&=& Var (Z' \Sigma Z') + 4 Tr^2(\Sigma)\\
&=& 8Tr(\Sigma^2) + 4\Delta_4 Tr(\Sigma \circ \Sigma) + 8 \delta^T \Sigma \delta+ 4Tr^2(\Sigma) + 4\|\delta\|^2 Tr(\Sigma) + \|\delta\|^4
\end{eqnarray*}

\end{proof}





\subsection*{Step (iii): The Behavior of $s_4 = Var(U_4)$}

We use the variance formula using the Hoeffding decomposition of the U-statistic $U_4$. We ignoring constants since we only aim to show that $Var(U_4/\gamma^4)$ is dominated by (is an order of magnitude smaller than) $Var(U_{CQ}/\gamma^2)$. Hence, we have by Lemma A of Section 5.2.1 of \cite{serfling},
\begin{equation}\label{eq:serfling-variance}
Var(U_4) \asymp \frac{Var(h_4)}{n^2} + \frac{Var(\E[h_4\vert  X,Y])}{n}.
\end{equation}

Some tedious algebra is required to estimate the second term. Recall that
\begin{eqnarray*}
U_4 &:=& \frac1{\binom{n}{2}} \sum_{i \neq j} h_4(X_i,X_j,Y_i,Y_j), \\
h_4(X_i,X_j,Y_i,Y_j) &:=& \|X_i-X_j\|^4 + \|Y_i - Y_j\|^4 - \|X_i - Y_j\|^4 - \|X_j - Y_i\|^4,\\
\theta &:=& \E\|X_i-X_j\|^4 + \E\|Y_i - Y_j\|^4 - \E\|X_i - Y_j\|^4 - \E\|X_j - Y_i\|^4.
\end{eqnarray*}
where $X,X' \sim P$ and $Y,Y' \sim Q$ from the model in [A1,A2] given by $X = \Gamma Z_1 $ and $Y = \Gamma Z_2 + \delta$.
 (since $h_4$ depends only on differences, we have assumed $\delta_1=0$ and $\delta_2 = \delta$ without loss of generality). Firstly, it is easy to verify that $h_4$ is a degenerate U-statistic under the null, since $\E [h_4 \vert  (X,Y)] = 0$ when $P=Q$.  We will now derive the variance of $\E [h_4 \vert  (X,Y)]$ when $P \neq Q$ under our assumptions.
Let us first derive $\E [h_4 \vert  (X,Y)]$ below. For convenience of notation, denote 
$$Y = \Gamma Z_Y$$ 
where $Z_Y = Z_2 + \eta$ and $\Gamma \eta = \delta$. Then
\begin{eqnarray*}
\|X-Y'\|^4 &=& (X^TX + Y'^T Y' - 2 X^T Y')^2 = (X^T X)^2 + (Y'^TY')^2 + 4(X^T Y')^2 \\
&& + 2 X^T X Y'^T Y' - 4 Y'^TY' X^T Y' - 4X^T X X^T Y',\\
\E[\|X-Y'\|^4 \vert  (X,Y)] &=& (X^T X)^2 +  \E [(Z_Y'^T \Pi Z_Y')^2]  + 4 X^T (\Sigma + \delta\delta^T) X + 2 X^T X (Tr(\Sigma) +\|\delta\|^2) \\
&& - 4\E[ Z_Y'^T \Pi Z_Y' Z_Y'^T \Gamma^T] \Gamma Z_1 - 4 X^T X X^T\delta,\\
\|X'-Y\|^4 &=& (X'^TX' + Y^T Y - 2 X'^T Y)^2 = (X'^T X')^2 + (Y^TY)^2 + 4(X'^T Y)^2 \\
&& + 2 X'^T X' Y^T Y - 4 Y^TY X'^T Y - 4X'^T X' X'^T Y,\\
\E[\|X'-Y\|^4 \vert  (X,Y)] &=& \E [(Z_1'^T \Pi Z_1')^2] + (Y^T Y)^2 + 4 Y^T \Sigma Y + 2 Y^T Y Tr(\Sigma)\\
&& - 4\E[ Z_1'^T \Pi Z_1' Z_1'^T \Gamma^T] (\Gamma Z_2 + \delta).
\end{eqnarray*}
Denoting $a_Y^T := \E[ Z_Y^T \Pi Z_Y Z_Y^T]$, we have
\begin{eqnarray*}
a_{Yk} &=& \E [ (\sum_{i \neq j} \Pi_{ij} Z_{Yi} Z_{Yj} + \sum_i \Pi_{ii} Z_{Yi}^2) Z_{Yk}] \\
&=& \E \left[ \sum_{i \neq j} \Pi_{ij} (Z_{2i} Z_{2j} + \eta_j Z_{2i} + \eta_i Z_{2j} + \eta_i\eta_j)(Z_{2k} + \eta_k) \right] \\
&& \quad + \E \left[\sum_i \Pi_{ii} (Z_{2i}^2 + 2Z_{2i} \eta_i + \eta_i^2) (Z_{2k} + \eta_k) \right]\\
&=& \left[ 0+ 0+ \sum_{j \neq k} \Pi_{kj} \eta_j  + 0 + \sum_{i \neq k} \Pi_{ik} \eta_i  + 0 +0 + \textcolor{blue}{\eta_k \sum_{i \neq j} \eta_i \Pi_{ij} \eta_j} \right] \\
&&  + \left[ \Delta_3 \Pi_{kk} + \eta_k \sum_i \Pi_{ii} + 2\Pi_{kk} \eta_k + 0 + 0 + \textcolor{blue}{\eta_k \sum_i \eta_i \Pi_{ii} \eta_i} \right]\\
&=& \left[ \textcolor{red}{2\sum_{j \neq k} \Pi_{jk} \eta_j }  \right] + \Bigg[ \Delta_3 \Pi_{kk} +\eta_k Tr(\Pi) + \textcolor{red}{2 \Pi_{kk} \eta_k} \Bigg] + \textcolor{blue}{\eta_k(\eta^T \Pi \eta)}\\
&=& \Delta_3 \Pi_{kk} + \eta_k Tr(\Pi) +   \textcolor{red}{2\Pi_k \eta} + \textcolor{blue}{\eta_k \|\delta\|^2 }.
\end{eqnarray*} 
Since $\Pi \eta = \Gamma^T \Gamma \eta = \Gamma^T\delta$, we have $a_Y^T = \Delta_3 diag(\Pi) + \eta Tr(\Pi) + 2\Gamma^T \delta + \|\delta\|^2\eta$.
Using this and calling $a_X^T = \E[Z_1^T \Pi Z_1 Z_1^T] = \Delta_3 diag(\Pi)$,
\begin{eqnarray*}
-\E[\|X-Y'\|^4 \vert  (X,Y)]
&=& \textcolor{blue}{-(X^T X)^2-\E [(Z_Y'^T \Pi Z_Y')^2]} \textcolor{magenta}{ - 4 X^T \Sigma X} - 4 X^T \delta\delta^T X \textcolor{magenta}{- 2 X^T X Tr(\Sigma)} \\
&& - 2 X^T X \|\delta\|^2 \textcolor{red}{ + 4 a_X^T \Gamma^T X} + 4Tr(\Sigma) \delta^T X +  8 \delta^T \Sigma X + 4 \|\delta\|^2 \delta^T X + 4 X^T X X^T\delta,\\
-\E[\|X'-Y\|^4 \vert  (X,Y)]
&=&  \textcolor{blue}{-\E [(Z_1'^T \Pi Z_1')^2]-(Y^T Y)^2} \textcolor{magenta}{- 4 Y^T \Sigma Y - 2 Y^T Y Tr(\Sigma)} \textcolor{red}{+ 4a_X^T \Gamma^T Y}, \\
\E[\|Y-Y'\|^4 \vert  (X,Y)]
&=& \textcolor{blue}{(Y^T Y)^2 +  \E [(Z_Y'^T \Pi Z_Y')^2]}  \textcolor{magenta}{+ 4 Y^T \Sigma Y} + 4Y^T \delta\delta^T Y  \textcolor{magenta}{+2 Y^T Y Tr(\Sigma)} +2Y^TY \|\delta\|^2  \\
&& \textcolor{red}{- 4a_X^T \Gamma^T Y}  - 4Tr(\Sigma) \delta^T Y - 8 \delta^T \Sigma Y - 4 \|\delta\|^2 \delta^T Y - 4 Y^TY Y^T\delta,\\
\E[\|X-X'\|^4 \vert  (X,Y)]
&=&  \textcolor{blue}{\E [(Z_1'^T \Pi Z_1')^2]+(X^T X)^2} \textcolor{magenta}{+4 X^T \Sigma X + 2 X^T X Tr(\Sigma)}  \textcolor{red}{-4a_X^T \Gamma^T X}.
\end{eqnarray*}
Adding the above 4 equations, we get
\begin{eqnarray}
\E [h_4 \vert  (X,Y)] &=&  4 \delta^T (YY^T - XX^T) \delta   + 2 (Y^T Y - X^T X) \|\delta\|^2 - 4 Tr(\Pi) \delta^T (Y-X) \nonumber \\
&& - 8 \delta^T \Sigma (Y-X) - 4\|\delta\|^2\delta^T(Y-X) -4 (Y^T Y Y^T - X^T X X^T)\delta \label{eq:Eh4Gxy}.
\end{eqnarray}

We will now take a detour to calculate the expectations and variances of products of quadratic forms, to aid us in bounding $Var(\E[h_4\vert (X,Y)])$ by bounding the variances of each term in Eq.\eqref{eq:Eh4Gxy} above.

\begin{proposition}
Let 
$Q := \epsilon^T \Pi \epsilon$ be a quadratic form, where $\epsilon$ is standard normal. Then
\begin{eqnarray*}
\E [Q] &=& Tr(\Pi)\\
\E [Q^2] &=& Tr^2(\Pi) + 2 Tr(\Pi^2)\\
Var(Q) &=& 2Tr(\Pi^2)\\
\E [Q^3] &=& Tr^3(\Pi) + 6 Tr(\Pi^2)Tr(\Pi) + 8Tr(\Pi^3)\\
\E [Q^4] 
&=& Tr^4(\Pi) + 12 Tr(\Pi^2)Tr^2(\Sigma) + 12 Tr^2(\Pi^2) + 32 Tr(\Pi)Tr(\Pi^3) + 48 Tr(\Pi^4)\\
Var(Q^2)
&=& Tr^4(\Pi) + 12 Tr(\Pi^2)Tr^2(\Pi) + 12 Tr^2(\Pi^2) + 32 Tr(\Pi)Tr(\Pi^3) + 48 Tr(\Pi^4) \\
&& - \left( Tr^4(\Pi) + 4Tr^2(\Pi^2) + 4 Tr(\Pi^2)Tr^2(\Pi) \right)\\
&\leq& 96 Tr(\Pi^2) Tr^2(\Pi)
\end{eqnarray*}
\end{proposition}
\begin{proof}

The expectations follow directly from the results of \cite{magnus79} and \cite{kendall1977advanced}. The last equation follows since $Tr(AB) \leq Tr(A) Tr(B)$ for any two psd matrices we have
$
Tr(\Pi^2) \leq Tr^2(\Pi)
$
and
$Tr(\Pi^3) \leq Tr(\Pi^2)Tr(\Pi)$
and
$Tr(\Pi^4) \leq Tr(\Pi^2)Tr^2(\Pi)$.
by Cauchy-Schwarz.

\end{proof}

\begin{proposition}\label{prop:ugly}
Let $Ts(A) = \sum_{ij} A_{ij}$ denote the \textit{T}otal \textit{s}um of all entries of $A$ and let $\circ$ denote Hadamard product. Let $Q = \epsilon^T \Pi \epsilon$, where  the moments of the coordinates of $\epsilon$ are given by
\begin{eqnarray*}
 m_1 &=& 0, \\
 m_2 &=& 1, \\
 m_3 &=& \Delta_3, \\
 m_4 &=& 3 + \Delta_4,\\
 m_5 &=& \Delta_5 + 10 \Delta_3,\\
 m_6 &=& \Delta_6 + 15 \Delta_4 + 10 \Delta_2^2 + 15,\\ 
 m_7 &=& \Delta_7 + 21 \Delta_5 + 35 \Delta_4\delta_3 + 105 \Delta_3,\\
 m_8 &=& \Delta_8 + 28 \Delta_6 + 56 \Delta_5 \Delta_3 + 35 \Delta_4^2 + 210 \Delta_4 + 280 \Delta_3^2 + 105.
\end{eqnarray*}
Here the $\Delta$s should be thought of as deviations from normality. $\Delta_3$ is skewness and $\Delta_4$ is kurtosis, and $\Delta_i=0$ for all $i$ if $\epsilon$ was standard Gaussian. Then, we have 
\begin{eqnarray*}
\E[Q] &=& Tr(\Pi),\\
Var[Q] &=& 2Tr(\Pi^2) + \Delta_4 Tr(\Pi \circ \Pi),\\ 
\E[Q^2]&=& 2Tr(\Pi^2) + \Delta_4 Tr(\Pi \circ \Pi) + Tr^2(\Pi),\\
\E[Q^4] &=& Tr^4(\Pi) + 12 Tr(\Pi^2)Tr^2(\Pi) + 12 Tr^2(\Pi^2) + 32 Tr(\Pi)Tr(\Pi^3) + 48 Tr(\Pi^4),\\
&& + \Delta_4 f_2 + \Delta_6 f_4 + \Delta_8 f_6 + \Delta_3^2 f_3 + \Delta_4^2 f_{42} + \Delta_3\Delta_5 f_{35}\\
\text{where } f_4 &=& 6 Tr^2(\Pi)Tr(\Pi \circ \Pi) + 12 Tr(\Pi^2)Tr(\Pi \circ \Pi) + 48 Tr(\Pi) Tr(\Pi \circ \Pi^2) \\
&& + 96 Tr(diag(\Pi) \Pi^3) + 48 Tr(diag^2(\Pi^2)),\\
f_6 &=& 4 Tr(\Pi)Tr(\Pi \circ \Pi \circ \Pi) + 24 Tr(\Pi \circ \Pi \circ \Pi^2),\\
f_8 &=& Tr(\Pi \circ \Pi \circ \Pi \circ \Pi),\\
f_3 &=& 24 Ts(diag(\Pi)\Pi diag(\Pi))Tr(\Pi) + 48 Ts(diag(\Pi)\Pi^2 diag(\Pi)) + 16 Ts(\Pi \circ \Pi \circ \Pi)Tr(\Pi)\\
&& + 96 Ts((\Pi \circ \Pi)\Pi diag(\Pi) ) + 96 Tr(\Pi (\Pi \circ \Pi) \Pi),\\
f_{42} &=& 3 Tr^2(\Pi \circ \Pi) + 24 Ts(diag(\Pi) (\Pi \circ \Pi) diag(\Pi)) + 8 Ts(\Pi \circ \Pi \circ \Pi \circ \Pi),\\
 f_{35} &=& 24 Ts(diag(\Pi )\Pi diag^2(\Pi)) + 32 Ts(diag(\Pi) (\Pi \circ \Pi \circ \Pi)),\\
 Var(Q^2) &\asymp& Tr(\Pi^2)Tr^2(\Pi).
\end{eqnarray*}
\end{proposition}

\begin{proof}
The first four claims follow directly from the detailed work of \cite{baoullah10}. Let us see how the last claim then follows. First note that  $Tr(\Pi \circ \Pi) \leq Tr(\Pi^2) \leq Tr^2(\Pi)$. The first inequality follows  because $\sum_{i} \Pi_{ii}^2 \leq \sum_{i,j} \Pi_{i,j}^2 = \|\Pi\|_F^2 = Tr(\Pi^2)$. The second follows because $0 \leq Tr(\Pi^2) = \langle \Pi, \Pi \rangle \leq \|\Pi\|_{op} \|\Pi\|_{*} \leq Tr^2(\Pi)$ by Cauchy-Schwarz. We also use the Hadamard product identity $diag(\Pi)(\Pi \circ \Pi) diag(\Pi) = (diag(\Pi)\Pi) \circ (\Pi diag(\Pi)) = (\Pi diag(\Pi)) \circ (diag(\Pi)\Pi) = \Pi \circ (diag(\Pi)\Pi diag(\Pi))$, see \cite{horn1991topics}.
Since $Tr(AB) \leq Tr(A) Tr(B)$ for any two psd matrices, we similarly have
\begin{align*}
 Ts(\Pi \circ \Pi \circ \Pi) &= \sum_{ij} \Pi_{ij}^3 \leq  \sum_{ij} \vert \Pi_{ij}\vert ^3 \leq  (\sum_{ij} \Pi_{ij}^2)^{3/2} = Tr^{3/2}(\Pi^2) \leq Tr(\Pi^2)Tr(\Pi)\\
 Tr(\Pi \circ \Pi \circ \Pi) &= \sum_{i} \Pi_{ii}^3 \leq (\sum_i \Pi_{ii}^2)^{3/2} \leq (\sum_{ij} \Pi_{ij}^2)^{3/2} < Tr(\Pi^2)Tr(\Pi)\\
Ts(\Pi \circ \Pi \circ \Pi \circ \Pi) &= \sum_{ij} \Pi_{ij}^4 = \langle \Pi \circ \Pi, \Pi \circ \Pi \rangle \leq  Tr^2(\Pi \circ \Pi) < Tr(\Pi^2)Tr^2(\Pi)\\
Tr(\Pi \circ \Pi \circ \Pi \circ \Pi) &< Tr(\Pi^2)Tr^2(\Pi)\\
Tr(diag(\Pi)\Pi^3) &\leq Tr(diag(\Pi)) Tr(\Pi^3) \leq Tr(\Pi^2)Tr^2(\Pi)\\
Tr(\Pi (\Pi \circ \Pi) \Pi) &\leq Tr(\Pi)Tr(\Pi \circ \Pi)Tr(\Pi) \leq Tr(\Pi^2)Tr^2(\Pi)\\
Ts(diag(\Pi)(\Pi \circ \Pi) diag(\Pi)) &\leq Tr^2(\Pi) Tr(\Pi^2).
\end{align*}

In this fashion, we can verify that the dominant term of $Var(Q^2)$ scales as $Tr(\Pi^2)Tr^2(\Pi)$.

\end{proof}


We can now extend these results to the case where the quadratic form is uncentered. 

\begin{proposition}\label{prop:uncenteredQ}
$Q = \epsilon^T \Pi \epsilon$ and $Q' = Q + a^T \epsilon + b$, where $\epsilon$ satisfies the conditions of the previous proposition, $a^T a = 4\delta^T \Sigma \delta$ and $b=\delta^T \delta$. Then
\begin{eqnarray*}
\E[Q'] &=& Tr(\Pi) + b\\
Q'^2 &=& Q^2 + (a^T \epsilon)^2 + b^2 + 2 Q a^T \epsilon + 2 b a^T \epsilon + 2 b Q\\
\E Q'^2 &\asymp& Tr^2(\Pi) + 2 Tr(\Pi^2) + a^T a + b^2 + 2 \Delta_3 diag(\Pi) a + 2b Tr(\Pi) \\
Var(Q') &\asymp& 2 Tr(\Pi^2) + a^T a + 2\Delta_3 diag(\Pi) a\\
Var(Q'^2) &\leq& 2 Var(Q^2) + 4 (a^Ta)^2 + 2\Delta_4 Tr(aa^T \circ aa^T) + 4 Var(Q a^T \epsilon) \\
&& + 4 b^2 a^T a + 8b^2 Tr(\Pi^2) + 4b^2 \Delta_4 Tr(\Pi \circ \Pi)\\
&\asymp& Tr^2(\Pi)Tr(\Pi^2)\\
&\asymp& Var(Q^2).
\end{eqnarray*}
\end{proposition}
\begin{proof}
All statements hold simply by expansion and substitution from the previous proposition. 
Remembering that $Var(Q^2) \asymp Tr(\Sigma^2)Tr^2(\Sigma)$, we can see that the last claim holds. Indeed, Assumption [A4] implies that $a^T a = o(\lambda_{\max}(\Sigma)Tr(\Sigma))$ and hence $(a^Ta)^2 = o(Tr(\Sigma^2)Tr^2(\Sigma))$ since $\lambda_{\max}^2(\Sigma) \leq \|\Sigma\|_F^2 = Tr(\Sigma^2)$. Similarly, $b^2 a^T a = o(Tr^2(\Sigma)Tr(\Sigma^2))$. In this fashion we deduce that the dominant term in $Var(Q'^2)$ is $Var(Q^2)$.

Since $Var(A+B) \leq 2 Var(A) + 2 Var(B)$ and $(a+b+c)^2 \leq 3a^2 + 3b^2 + 3c^2$, we can alternately derive the following bound for variances of quadratic forms involving $Y = \Gamma Z_2 + \delta$:
\begin{align*}
Y^T Y ~&=~ Z_2^T \Pi Z_2 + \delta^T \delta + 2\delta^T \Gamma Z_2\\
Y^T \Sigma Y ~&=~ Z_2^T \Pi^2 Z_2 + \delta^T \Sigma \delta + 2 \delta^T \Sigma \Gamma Z_2\\
(Y^T Y)^2  ~&\leq~ 3(Z_2^T \Pi Z_2)^2 + 3(\delta^T\delta)^2 + 3 (\delta^T \Gamma Z_2)^2\\
\E[ Y^T Y ] ~&=~ Tr(\Sigma) + \delta^T \delta\\
\E [Y^T \Sigma Y ] ~&=~ Tr(\Sigma^2) + \delta^T \Sigma \delta\\
Var(Y^T Y) ~&\leq~ 4 Tr(\Sigma^2) + 8 \delta^T \Sigma \delta\\
\E[(Y^TY)^2] ~&=~ Var(Y^T Y) + \E^2(Y^T Y) \\ 
~&\leq~  4Tr(\Sigma^2) + 8 \delta^T \Sigma \delta + (Tr(\Sigma) + \delta^T \delta)^2 ~\asymp~ Tr^2(\Sigma)\\
Var(Y^T \Sigma Y) ~&\leq~ 4 Tr(\Sigma^4) + 8 \delta^T \Sigma^3 \delta\\
Var((Y^T Y)^2) ~&\leq~ 18 Var((Z^T \Pi Z_2)^2) + 18 Var((\delta^T  \Gamma Z_2)^2)\\
~&\asymp~ Tr(\Sigma^2)Tr^2(\Sigma) + (\delta^T \Sigma \delta)^2
\end{align*}
where we used $var((v^T Z)^2) = var(Z^T vv^T Z) = 2Tr((vv^T)^2) = 2(v^T v)^2$. Since $\delta^T \Sigma \delta = o(Tr(\Sigma^2))$ by our assumptions, the last expression is dominated by its first term. 

\end{proof}
%
%
%


\begin{proposition}\label{prop:VarYtYYtD}
\begin{eqnarray*}
Var(X^TXX^T\delta) &\asymp& Tr^2(\Sigma)\delta^T \Sigma \delta\\
Var(Y^T Y Y^T \delta) &\asymp& Tr^2(\Sigma)\delta^T \Sigma \delta\\
&\asymp& Var(X^TXX^T\delta).
\end{eqnarray*}
\end{proposition}
\begin{proof}
Let us first calculate $Var(X^T X X^T \delta)$, for which we need to know $\E[XX^T XX^T XX^T]$. Let us first calculate $\E[XX^T XX^T]$.
For this purpose, see that $\E(Z_1 Z_1^T \Pi  Z_1Z_1^T) = \E( (Z_1^T \Pi  Z_1) Z_1 Z_1^T) =   2\Pi + Tr(\Pi) I$.
This is true because its off-diagonal element is $\E (\sum_{ij} \Pi_{ij} z_i z_j z_a z_b) = 2\Pi_{ab}$, and its diagonal is $ \E (\sum_{ij} \Pi_{ij} z_i z_j z_a^2 ) = 3 \Pi_{aa} + \sum_{k \neq a} \Pi_{kk} = Tr(\Pi) + 2\Pi_{aa}$.
Hence $\E(XX^T X X^T) = \Gamma  \E (Z_1 Z_1^T \Pi Z_1 Z_1^T ) \Gamma^T = 2 \Sigma^2 + Tr(\Sigma) \Sigma$. Now, we are ready to calculate $\E[XX^T XX^T XX^T]$.
\begin{flalign*}
\text{Define } C &:= \E((Z_1^T \Pi Z_1)^2 Z_1Z_1^T)\\
\text{Hence } C_{aa} &= \E (\sum_{ijkl} \Pi_{ij}\Pi_{kl} z_i z_j z_k z_l z_a^2) \\
&= 15 \Pi_{aa}^2 + 6\Pi_{aa} (\sum_{t \neq a} \Pi_{tt}) + 12 \sum_{t \neq a} \Pi_{ta}^2 + 3 \sum_{t \neq a} \Pi_{tt}^2 + 2\sum_{s \neq t \neq a} \Pi_{ss}\Pi_{tt} + 4 \sum_{s \neq t \neq a }\Pi^2_{st}
\end{flalign*}
Let us simplify this expression. Notice the following identities:
\begin{flalign*}
2Tr(\Pi^2) &= 2\Pi_{aa}^2 + 4\sum_{t \neq a} \Pi_{ta}^2 + 2\sum_{t \neq a} \Pi_{tt}^2 + 4 \sum_{s \neq t \neq a} \Pi_{st}^2 \\
Tr^2(\Pi) &= \Pi_{aa}^2 + \sum_{t \neq a} \Pi_{tt}^2 + 2 \sum_{t \neq a } \Pi_{tt} \Pi_{aa} + 2\sum_{s\neq t \neq a} \Pi_{ss} \Pi_{tt}\\
8\Pi_{.a}^T \Pi_{.a} &= 8\Pi_{aa}^2 +8 \sum_{t \neq a} \Pi_{ta}^2\\
4Tr(\Pi) \Pi_{aa} &= 4\Pi_{aa}^2 + 4\sum_{t \neq a} \Pi_{tt}\Pi_{aa}\\
\text{Hence, we see that } C_{aa} &= 6 \Pi_{aa}^2 + 4 Tr(\Pi)\Pi_{aa} + 2 (\Pi^2)_{aa} + 2Tr(\Pi^2) + Tr^2(\Pi)\\
\text{Similarly } C_{ab} &= \E (\sum_{ijkl} \Pi_{ij}\Pi_{kl} z_i z_j z_k z_l z_a z_b) \\
&= 8 \sum_{t \neq a \neq b} \Pi_{at} \Pi_{bt} + 4 \sum_{t \neq a \neq b} \Pi_{ab} \Pi_{tt} + 12 \Pi_{aa}\Pi_{ab} + 12 \Pi_{bb} \Pi_{ab}\\
&= 4 \Pi_{ab} Tr(\Pi) + 8(\Pi^2)_{ab}\\
\text{Hence } C &= 8\Pi^2 + 4Tr(\Pi) \Pi + (2Tr(\Pi^2) + Tr^2(\Pi))I \\
\end{flalign*}
Hence 
\begin{flalign} 
\E[XX^T XX^T XX^T] &= 8\Sigma^3 + 4 Tr(\Sigma)\Sigma^2 + 2 Tr(\Sigma^2)\Sigma +Tr^2(\Sigma)\Sigma \label{eq:EXXtXXtXXt} \\
\text{and ~}~  Var(X^T X X^T \delta) &\asymp \delta^T \Sigma^3 \delta  + Tr(\Sigma)\delta^T \Sigma^2 \delta + Tr^2(\Sigma) \delta^T \Sigma \delta \nonumber\\
&\asymp \textcolor{red}{Tr^2(\Sigma) \delta^T \Sigma \delta}. \nonumber
\end{flalign}
Next, let us calculate $Var(Y^T Y Y^T \delta)$. We keep only the higher order terms in the following expansions, to avoid the tediousness of Proposition \ref{prop:ugly}  for clarity.
\begin{flalign*}
\E[YY^T] &= \Sigma + \delta\delta^T\\
\E(Y^T Y Y^T \delta) &= \E[(\Gamma Z_2 + \delta)^T (\Gamma Z_2 +
\delta) (Z_2^T \Gamma^T \delta + \delta^T \delta)]\\
&= \|\delta\|^2 (Tr(\Sigma)+\delta^T\delta) + 2\delta^T \Sigma \delta \\
&\asymp \textcolor{blue}{\|\delta\|^2Tr(\Sigma)}\\
\E [YY^T YY^T ] &= \E[ (\Gamma Z_2 + \delta)(\Gamma Z_2 + \delta)^T(\Gamma Z_2 + \delta)(\Gamma Z_2 + \delta)^T] \\
& \asymp \Gamma B \Gamma^T + \delta\delta^T (\Sigma + \delta\delta^T) + \delta (Tr(\Sigma) + \delta^T\delta) \delta^T + (\Sigma + \delta\delta^T)\delta \delta^T + \|\delta\|^2(\Sigma + \delta \delta^T) \\
& \quad + \E [\delta Z_2^T \Gamma^T \delta Z_2^T \Gamma^T] + \E[\Gamma Z_2 \delta^T \Gamma Z_2 \delta^T] + \|\delta\|^2 \delta\delta^T\\
\E[\delta^T YY^T YY^T \delta] &= 2 \delta^T \Sigma^2 \delta + Tr(\Sigma)\delta^T \Sigma \delta + 5 \|\delta\|^2 \delta^T \Sigma \delta + 5\|\delta\|^6 + \|\delta\|^4 Tr(\Sigma) \\
& \asymp \delta^T \Sigma \delta Tr(\Sigma) + \|\delta\|^4 Tr(\Sigma)\\
\E [\delta^T YY^T YY^T YY^T \delta] &= \delta^T \E[(\Gamma Z_2 + \delta)(\Gamma Z_2 + \delta)^T(\Gamma Z_2 + \delta)(\Gamma Z_2 + \delta)^T(\Gamma Z_2 + \delta)(\Gamma Z_2 + \delta)^T] \delta\\
& \asymp \|\delta\|^2 (\E[\delta^T YY^T YY^T \delta]) + \delta^T \E [\Gamma Z_2 Z_2^T \Gamma^T YY^T YY^T ]\delta \\
& \quad + \E[\delta^T \Gamma Z_2 \delta^T YY^T YY^T] \delta + \|\delta\|^2\E[ Z_2^T \Gamma^T YY^T YY^T] \delta\\
&:= G_1 + G_2 + G_3 + G_4\\
\text{Define } \Phi &:= \Gamma^T \delta \delta^T \Gamma, \text{ and let us expand the 4 terms above.}\\
G_2 = \delta^T \E [\Gamma Z_2 Z_2^T \Gamma^T YY^T YY^T ]\delta &= \textcolor{red}{\delta^T \E [XX^T XX^T XX^T]\delta} + \|\delta\|^2 \delta^T \E[XX^T XX^T]\delta + 3 \|\delta\|^2 \E [Z_2^T \Phi Z_2 Z_2^T \Pi Z_2] \\
& \quad + 2\E [(Z_2^T \Phi Z_2)^2 ] + \E [ Z_2^T \Phi Z_2 ]  \|\delta\|^4  \\
& \asymp \delta^T \Sigma^3 \delta  + Tr(\Sigma)\delta^T \Sigma^2 \delta + \textcolor{red}{ Tr^2(\Sigma) \delta^T \Sigma \delta} + \|\delta\|^2 \delta^T \Sigma^2 \delta + \|\delta\|^2 \delta^T\Sigma \delta Tr(\Sigma) \\
& \quad +  (\delta^T \Sigma \delta )^2 +  \delta^T\Sigma\delta\|\delta\|^4 \\
&\asymp \textcolor{red}{Tr^2(\Sigma) \delta^T \Sigma \delta}\\
G_1  = \|\delta\|^2 (\E[\delta^T YY^T YY^T \delta]) &= \|\delta\|^6  Tr(\Sigma) + \|\delta\|^2 \delta^T \Sigma \delta Tr(\Sigma) \\
&\preceq G_2\\
G_3 = \E[\delta^T \Gamma Z_2 \delta^T YY^T YY^T] \delta
 &=  2\E[Z_2^T \Phi Z_2 Z_2^T \Pi Z_2 ] \|\delta\|^2 + 2\E[(Z_2^T \Phi Z_2)^2] + 4\E[Z_2^T \Phi Z_2]\|\delta\|^4 \\
 &\asymp \|\delta\|^2 \delta^T\Sigma \delta Tr(\Sigma)+ \|\delta\|^2\delta^T \Sigma^2 \delta +  (\delta^T \Sigma \delta )^2 +  \delta^T\Sigma\delta\|\delta\|^4\\
 &\preceq G_2 \\
G_4 =  \|\delta\|^2\E[ Z_2^T \Gamma^T YY^T YY^T] \delta
 &= \|\delta\|^4 \E[(Z_2^T \Pi Z_2)^2 ]  + 3\|\delta\|^2 \E[Z_2^T \Pi Z_2 Z_2^T \Phi Z_2]  \\ 
 & \quad ~ +  \|\delta\|^6 \E[Z_2^T \Pi Z_2 ] + 3\|\delta\|^4 \E[Z_2^T \Phi Z_2]  \\
 &\asymp \|\delta\|^4 Tr^2(\Sigma) + \|\delta\|^2 \delta^T \Sigma \delta Tr(\Sigma) + \|\delta\|^2 \delta^T\Sigma^2 \delta + \|\delta\|^6 Tr(\Sigma) + \|\delta\|^4 \delta^T \Sigma \delta\\
  &\asymp \|\delta\|^4 Tr^2(\Sigma)\\
 \text{Hence } Var(Y^TY Y^T \delta) &= \E[\delta^T Y Y^T Y Y^T Y Y^T \delta] - \E^2[Y^T YY^T \delta] \\
 &\asymp G_1 + G_2 + G_3 + G_4 - \textcolor{blue}{\|\delta\|^4Tr^2(\Sigma)}\\
 &\asymp \textcolor{red}{Tr^2(\Sigma) \delta^T \Sigma \delta}\\
 &\asymp Var(X^T X X^T \delta)
\end{flalign*}
\end{proof}

\begin{lemma}
$$
Var(\E[h_4\vert (X,Y)])  \asymp Tr^2(\Sigma) \delta^T \Sigma \delta
$$
\end{lemma}
\begin{proof}

Returning back to  Eq.\eqref{eq:Eh4Gxy}, the 4 different variance terms involved in $Var(\E [h_4 \vert  (X,Y)])$ are
\begin{eqnarray*}
Var(Y^T \delta \delta^T Y) &=& Var((\Gamma Z_2 + \delta)^T\delta\delta^T(\Gamma Z_2 + \delta)) \asymp  (\delta^T \Sigma \delta)^2 + \|\delta\|^4 \delta^T \Sigma \delta\\
Var(Y^T Y \|\delta\|^2) &\asymp& \|\delta\|^4 Tr(\Sigma^2) \\
Var(Tr(\Pi) \delta^T \Gamma (Z_2 - Z_1)) &\asymp& Tr^2(\Sigma) \delta^T \Sigma \delta\\
Var(Y^T Y Y^T \delta) &\asymp& Tr^2(\Sigma)\delta^T \Sigma \delta
\end{eqnarray*}

Under our assumptions, one can verify that the dominant term of $Var(\E[h_4\vert X,Y])$ is $\asymp Tr^2(\Sigma) \delta^T \Sigma \delta$.
\end{proof}

\begin{lemma}
$$
Var(h_4) \asymp Tr^2(\Sigma)Tr(\Sigma^2)
$$
\end{lemma}
\begin{proof}
\begin{eqnarray}
h_4 &=& 4 [ (X^T X')^2 + (Y^T Y')^2 - (X^TY')^2 - (X'^T Y)^2 ] \nonumber \\
&+& 2 [ X^TX (X'^T X' - Y'^T Y') + Y^TY(Y'^T Y' - X'^T X') ] \nonumber\\
&+& 4 [ Y'^T Y' Y'^T (X - Y) + X'^T X' X'^T(Y-X) + X^T X X^T(Y'-X') + Y^T Y Y^T (X'-Y')] \nonumber \\
\label{eq:h4expansion}
\end{eqnarray}
For example, let us calculate $Var((X^T X')^2)$. Defining $S' = X'X'^T$, we have
\begin{flalign*}
\E[(X^T X')^4] &= \E_{X'}  \E_{X}[(X^T S' X)^2] = \E_{X'} \E_{Z_1} [(Z_1^T \Gamma^T S' \Gamma Z_1)^2]\\
 &= \E_{X'} [Tr(\Gamma^T X' X'^T \Gamma\Gamma^T X' X'^T \Gamma)  + Tr^2(\Gamma^T X' X'^T \Gamma)]\\
 &= \E_{X'} [ (X'^T \Sigma X')^2 + (X'^T \Sigma X')^2 ]\\
 &= \E_{X'} [ (Z_1'^T \Pi^2 Z_1')^2 + (Z_1'^T \Pi^2 Z_1')^2 ]\\
 &= 2Tr(\Pi^4) + Tr^2(\Pi^2) \\
 \E[(X^T X')^2] &= \E_{X'} \E_{X}[ Z_1^T \Gamma^T S' \Gamma Z_1 ] = \E_{X'} Tr(\Gamma^T X'X'^T \Gamma) = \E_{Z_1'} Z_1'^T \Pi^2 Z_1'\\
 &= Tr(\Pi^2)\\
 Var((X^TX')^2) &= \E[(X^T X')^4] - \E[(X^T X')^2]^2 = Tr(\Pi^4) = Tr(\Sigma^4) = o(Tr^2(\Sigma)Tr(\Sigma^2)) 
\end{flalign*}

Similarly, let us calculate $Var(X'^T X' X^T X) $ and $Var(Y'^T Y' Y^T Y)$ as follows.
\begin{eqnarray*}
Var(X'^T X' X^T X) &=& \E[ (X^T X)^2 (X'^T X')^2 ] - \E^2[X^T X X'^T X'] \\
&=& \E^2[(X^TX)^2] - \E^4[X^T X] \asymp (8Tr(\Sigma^2) + 4Tr^2(\Sigma))^2 - (2Tr(\Sigma))^4\\
&\asymp& Tr(\Sigma^2) Tr^2(\Sigma)\\
\text{and } Var(Y'^T Y' Y^T Y) &=& \E^2[(Y^T Y)^2] - \E^4(Y^T Y)\\
 &=& (Tr^2(\Sigma) + 2Tr(\Sigma^2) + 4 \delta^T \Sigma \delta + \delta^T \delta  \\
&& \quad ~ + 8\Delta_3 diag(\Pi)\delta^T \Sigma \delta + 2\delta^T \delta Tr(\Sigma))^2 - (Tr(\Sigma) + \delta^T \delta)^4\\
 &\asymp& Tr^2(\Sigma)Tr(\Sigma^2)
\end{eqnarray*}
where we use Proposition \ref{prop:uncenteredQ} and the last step follows by larger terms canceling after direct expansion.

Next, let us bound $Var(X^T X X^T X')$ and $Var(Y^T Y Y^T Y')$ as follows (other terms are similar).
Multiplying Eq.\eqref{eq:EXXtXXtXXt} by $\Sigma$, we see that
$$
\E[XX^T XX^T XX^T \Sigma]  = 8\Sigma^4 + 4 Tr(\Sigma)\Sigma^3 + 2 Tr(\Sigma^2)\Sigma^2 +Tr^2(\Sigma)\Sigma^2 .
$$
Now taking traces on both sides, and applying trace rotation to the left, we see that the dominant term is
$$
Tr(\E[XX^T XX^T XX^T \Sigma]) = \E[Tr(X^T XX^T XX^T \Sigma X)] = \E[(X^T X)^2 X^T \Sigma X] \asymp Tr(\Sigma^2)Tr^2(\Sigma).
$$
Since $Var(P) \leq \E [P^2]$, we conclude that
$$
Var(X^T X X^T X') \leq \E[X^T X X^T (X'X'^T) X X^T X] = \E[X^T\Sigma X (X^TX)^2] \asymp Tr^2(\Sigma)Tr(\Sigma^2).
$$
Then, taking expectations with respect to $Y'$ first, we get 
\begin{flalign*}
Var(Y^T Y Y^T Y') &= \E[Y^T (\Sigma + \delta \delta^T) Y Y^T Y Y^T Y] - \E^2[Y^T Y Y^T\delta]\\
&= \E [Y^T \Sigma Y (Y^T Y)^2] + Var(Y^T Y Y^T \delta)\\
&\asymp \E[ Z_Y^T \Sigma^2 Z_Y (Z_Y^T \Sigma Z_Y)^2 ] + Tr^2(\Sigma) \delta^T \Sigma \delta\\
&\asymp (\delta^T\Sigma\delta)^2 \delta^T \Sigma^2 \delta + 4(\delta^T \Sigma^2 \delta)^2 + 8 (\delta^T\Sigma \delta )(\delta^T\Sigma^3 \delta) + 8 \delta^T \Sigma^3 \delta \\
&\quad  + 4Tr(\Sigma^2)[\delta^T\Sigma^2\delta + (\delta^T\Sigma \delta)^2] + 8Tr(\Sigma)[\delta^T \Sigma^3 \delta + (\delta^T \Sigma^2 \delta)(\delta^T \Sigma \delta)]\\
&\quad + 3Tr(\Sigma^2)\delta^T \Sigma^2 \delta + 6Tr(\Sigma)\delta^T \Sigma^3 \delta + Tr^2(\Sigma)Tr(\Sigma^2) \\
&\quad + 4 Tr(\Sigma^3)Tr(\Sigma) + 2Tr^2(\Sigma^2) + 8Tr(\Sigma^4)\\
&\asymp Tr^2(\Sigma)Tr(\Sigma^2).
\end{flalign*}
The above results are obtained in a fashion similar to Proposition \ref{prop:uncenteredQ} for variance of uncentered quadratic forms, or Proposition \ref{prop:VarYtYYtD} for $Var(Y^T YY^T \delta)$, or from the results of \cite{baoullah10} about momnents of products of non-normal quadratic forms (Pg. 255 of \cite{ullah} for the Gaussian case).
Hence, bounding the $Var(h_4)$ by (a constant times) the sum of variances of the terms in the expansion Eq.\eqref{eq:h4expansion}, we see that
$$
Var(h_4) \asymp Tr^2(\Sigma)Tr(\Sigma^2)
$$
as required, concluding the proof of the lemma.

\end{proof}

In summary, using Eq.\eqref{eq:serfling-variance}, we have the variance of $U_4$ as
$$
Var(U_4) \leq C_1 \frac{Tr(\Sigma^2)Tr^2(\Sigma)}{n^2} + C_2 \frac{Tr^2(\Sigma)\delta^T \Sigma \delta }{n} \leq C Tr^2(\Sigma) Var(U_{CQ})
$$
for some absolute constants $C_1,C_2,C=\max\{C_1,C_2\}$.

Since $\gamma^2 = \omega(Tr(\Sigma))$, we see that 
$$Var(U_4/\gamma^4) = o(Var(U_{CQ}/\gamma^2))$$
as required for step (iii).

\begin{remark}
 Recall that it is typically stated in textbooks like \cite{serfling}, that for degenerate U-statistics, the variance under the null is $O(1/n^2)$, and variance under the alternative is $O(1/n)$. While this is true asymptotically when $n \to \infty$ in the fixed $d$ setting, the variance under the alternative can still be $O(1/n^2)$ in the high-dimensional setting, depending on the signal to noise ratio and dimension when $d,n \to \infty$.
\end{remark}

The conclusion of step (iii) also concludes the proof of Theorem \ref{thm:gmmd}.

\subsection{Proof of Theorem \ref{thm:eed}}

The only difference from the above proof, is that instead of taking the Taylor expansion of the Gaussian kernel, we take the expansion of the (modified) Euclidean distance. This gives rise to the exact same set of terms to bound, with different constants. Indeed,
when $\gamma^2 = \omega(Tr(\Sigma))$, by the exact form of Taylor's theorem for $f(\cdot) = (1+ \cdot)^{1/2}$  at $a = \frac{\|S_i - S_j\|^2}{\gamma^2-2Tr(\Sigma)}$ around $\tau = \frac{2Tr(\Sigma)}{\gamma^2 - 2Tr(\Sigma)} = o(1)$,
\begin{equation}\label{eq:eedtaylor}
f(a) = f(\tau) + \frac{(a-\tau)}{2(1+\tau)^{1/2}} - \frac{(a-\tau)^2}{8(1+\tau)^{3/2}}  + \frac{3(a-\tau)^3}{48} (1+ \zeta)^{-5/2}
\end{equation}
for some $\zeta$ between $a$ and $\tau$. Comparing Eq.\eqref{eq:eedtaylor} with Eq.\eqref{eq:gmmdtaylor}, we see that all the terms are exactly the same, except for constants. Hence, exactly the same proof of Theorem \ref{thm:gmmd} goes through for Theorem \ref{thm:eed} as well.

\subsection*{Acknowledgments}

This project was supported by the grant NSF IIS-1247658.

\bibliography{mmd}
\bibliographystyle{plainnat}

\appendix

\section{An error in \cite{cq} : the power for high SNR}\label{appsec:cq}
We briefly describe an error in \cite{cq}, that has a few important repercussions. All notations, equation numbers and theorems in this paragraph refer to those in \cite{cq}. Using the test statistic $T_n/\hat \sigma_{n1}$ defined below Theorem 2 in \cite{cq}, we can derive the power under their assumption (3.5) as

\begin{eqnarray*}
&&P_1 \left( \frac{T_n}{\hat \sigma_{n1}} > \xi_\alpha \right) =\\
&=&P_1 \left( \frac{T_n - \|\mu_1 - \mu_2\|^2}{\hat \sigma_{n2}} > \frac{\hat \sigma_{n1}}{\hat \sigma_{n2}} \xi_\alpha -\frac{ \|\mu_1 - \mu_2\|^2}{\hat \sigma_{n2}} \right)\\
& \rightarrow& \Phi \left( \frac{ \|\mu_1 - \mu_2\|^2}{\hat \sigma_{n2}} \right) \mbox{ (the denominator is \textit{not} $\hat \sigma_{n1}$)}\\
 &=&\Phi \left( \frac{\sqrt{n} \|\mu_1-\mu_2\|^2}{\sqrt{(\mu_1 - \mu_2)^T \Sigma (\mu_1 - \mu_2)}} \right)
 \end{eqnarray*}
which should be the expression for power that they derive in Eq.(3.12), the most important difference being the presence of $\sqrt n$ instead of $n$ in the numerator. 

%
%
%
%

\end{document}